\documentclass{article}
\usepackage{amsmath}
\usepackage{amssymb}
\usepackage{amsthm}
\usepackage{bbm}
\usepackage{amscd}
\usepackage{amssymb}
\usepackage{amsthm}
\usepackage{graphicx}
\usepackage{color}
\usepackage[all]{xy}
\usepackage{mathrsfs}
\usepackage{xcolor} 
\definecolor{brightmaroon}{rgb}{0.76, 0.13, 0.28}
\usepackage[
  breaklinks,
  unicode,
  colorlinks=true,
  linkcolor=black,        
  citecolor=brightmaroon, 
  urlcolor=black          
]{hyperref}
\usepackage{mathtools}
\usepackage{yhmath}
\usepackage{tikz-cd}
\usepackage[
	backend=biber, 
	citestyle=numeric-comp,
	giveninits=true,
	date=year,
	sorting=nyt, 
	maxnames=7,
    isbn=false,
	bibstyle=joycebib
]{biblatex}
\theoremstyle{plain}
\newtheorem{thm}{Theorem}[section]
\newtheorem{lem}[thm]{Lemma}
\newtheorem{prop}[thm]{Proposition}
\newtheorem{cor}[thm]{Corollary}

\theoremstyle{definition}
\newtheorem{dfn}[thm]{Definition}

\newtheorem{ex}[thm]{Example}
\newtheorem{rem}[thm]{Remark}
\numberwithin{equation}{section}
\numberwithin{figure}{section}
\def\e#1\e{\begin{equation}#1\end{equation}}
\def\ea#1\ea{\begin{align}#1\end{align}}
\def\eq#1{{\rm(\ref{#1})}}
\def\dim{\mathop{\rm dim}\nolimits}
\def\Man{{\mathop{\bf Man}}}
\def\Manb{{\mathop{\bf Man^b}}}
\def\Manc{{\mathop{\bf Man^c}}}
\def\Mangc{{\mathop{\bf Man^{gc}}}}
\def\Mancin{{\mathop{\bf Man^c_{in}}}}  
\def\Mangcin{{\mathop{\bf Man^{gc}_{in}}}}
\def\codim{\mathop{\rm codim}\nolimits}
\newcommand{\gp}{\mathrm{gp}}
\newcommand{\sk}{\llcorner}
\def\KN{{\rm KN}}
\def\depth{\mathop{\rm depth}\nolimits}
\def\oM{{\mathbin{\smash{\,\,\overline{\!\!\mathcal M\!}\,}}}}
\def\Im{\mathop{\rm Im}\nolimits}

\def\Ker{\mathop{\rm Ker}}

\def\Re{\mathop{\rm Re}}

\def\supp{\mathop{\rm supp}}
\def\rank{\mathop{\rm rank}\nolimits}
\def\Hom{\mathop{\rm Hom}\nolimits}
\def\id{{\mathop{\rm id}\nolimits}}

\def\rst{{\rm st}}

\def\ge{\geqslant}
\def\le{\leqslant\nobreak}
\def\R{{\mathbin{\mathbb R}}}
\def\Z{{\mathbin{\mathbb Z}}}

\def\N{{\mathbin{\mathbb N}}}
\def\C{{\mathbin{\mathbb C}}}
\def\cS{{\mathbin{\cal S}}}
\def\al{\alpha}
\def\be{\beta}

\def\de{\delta}

\def\la{\lambda}

\def\th{\theta}

\def\vp{\varphi}
\def\si{\sigma}
\def\om{\omega}

\def\La{\Lambda}
\def\Ga{\Gamma}
\def\pd{\partial}
\def\db{\bar\partial}
\def\ts{\textstyle}
\def\w{\wedge}
\def\sm{\setminus}
\def\bu{\bullet}
\def\op{\oplus}
\def\ot{\otimes}
\def\ov{\overline}
\def\bigop{\bigoplus}

\def\iy{\infty}
\def\es{\emptyset}
\def\ra{\rightarrow}
\def\ab{\allowbreak}
\def\longra{\longrightarrow}
\def\hookra{\hookrightarrow}
\def\ha{{\ts\frac{1}{2}}}
\def\t{\times}
\def\ci{\circ}
\def\ti{\tilde}

\def\d{{\rm d}}
\def\an#1{\langle #1 \rangle}
\def\ban#1{\bigl\langle #1 \bigr\rangle}
\begin{document}

\title{B-complex manifolds with generalized corners. I. Newlander--Nirenberg Theorems}
\author{H\"ulya Arg\"uz and Dominic Joyce}
\date{}
\maketitle
\begin{abstract} 
We generalize complex manifolds to manifolds with corners $X$, and to manifolds with generalized corners (g-corners) in the sense of the second author \cite{Joyc2}, using complex structures on the b-tangent bundle (log tangent bundle) ${}^bTX$. We prove a formal Newlander--Nirenberg type theorem  showing that along each corner stratum of $X$, the b-complex structure agrees with a standard model to infinite order.

In the sequel \cite{ArJo2} we show that if $S$ is a log smooth log $\C$-scheme, or log smooth log complex analytic space, then the Kato--Nakayama space $S^\KN$ has the structure of a b-complex manifold with g-corners. Using our Newlander--Nirenberg theorem we give necessary and sufficient conditions for a b-complex manifold with g-corners to be a Kato--Nakayama space.
\end{abstract}

\setcounter{tocdepth}{2}
\tableofcontents

\section{Introduction}
\label{bc1}

\subsection{Overview}

Manifolds with corners, locally modelled on $[0,\iy)^k\t\R^{m-k}$, provide a natural extension of smooth manifolds that accommodates boundary and corner strata. They are particularly useful in situations where non-compact manifolds admit compactifications by adding boundary faces with controlled geometric structure, allowing one to study asymptotic behaviour within a differential-geometric framework. Manifolds with generalized corners (g-corners), introduced by the second author \cite{Joyc2}, further extend this notion by allowing more flexible local models governed by `weakly toric monoids'. 

The aim of this paper is to extend the theory of complex manifolds to the setting of manifolds with corners and, more generally, manifolds with g-corners. Classically, complex manifolds may be defined either via holomorphic coordinate charts or via integrable almost complex structures, with the equivalence of these approaches ensured by the Newlander–Nirenberg theorem. 

In this work, we develop a theory of b-complex manifolds with g-corners following the latter perspective, based on b-almost complex structures and their integrability. Our main result establishes a formal analogue of the Newlander–Nirenberg theorem in this context. 

A sequel \cite{ArJo2} applies these ideas to log geometry: we show that if $S$ is a log smooth log $\C$-scheme or log complex analytic space then the `Kato--Nakayama space' $S^\KN$ is a b-complex manifold with g-corners, and characterizes the b-complex manifolds which arise this way.

\subsection{Background and context}

We recall some well known material, see e.g.\ Kobayashi and Nomizu \cite[\S IX]{KoNo}.

\begin{dfn}
\label{bc1def1}
Let $X$ be a smooth manifold of real dimension $2n$, and $TX\ra X$ the tangent bundle of $X$. An {\it almost complex structure\/} $J$ is a morphism of vector bundles $J:TX\ra TX$ such that $J^2=-\id_{TX}$. We call $(X,J)$ an {\it almost complex manifold}.

A smooth function $f:X\ra\C$ is called {\it holomorphic\/} if $J(\d f)=i\,\d f$ in $\Ga^\iy(T^*X\ot_\R\C),$ where $\Ga^\iy(E)$ denotes the vector space of smooth sections of~$E$.

An almost complex structure has a {\it Nijenhuis tensor\/} $N_J\in\Ga^\iy(T^*X\ot T^*X\ot TX)$ characterized by the fact that for all vector fields $v,w\in\Ga^\iy(TX)$, we have 
\e
N_J(v,w)=[v,w]+J\bigl([Jv,w]+[v,Jw]\bigr)-[Jv, Jw],
\label{bc1eq1}
\e
where $[\,,\,]$ is the Lie bracket of vector fields. We call $J$ {\it integrable\/} if $N_J=0$, and then we call $(X,J)$ a {\it complex manifold}.

Equivalently, the complexification $TX\ot_\R\C$ has a splitting $TX\ot_\R\C=T^{1,0}X\op T^{0,1}X$, where $T^{1,0}X$ is the eigenspace of $J$ with eigenvalue $i$, and $T^{0,1}X$ the eigenspace of $J$ with eigenvalue $-i$. Then $J$ is integrable if and only if~$\bigl[\Ga^\iy(T^{1,0}X),\Ga^\iy(T^{1,0}X)\bigr]\subseteq \Ga^\iy(T^{1,0}X)\subseteq\Ga^\iy(TX\ot_\R\C)$.
\end{dfn}

Here is the {\it Newlander--Nirenberg Theorem\/} \cite{NeNi}, a central result.

\begin{thm}
\label{bc1thm1}
Let\/ $(X,J)$ be an almost complex manifold of real dimension $2n$. Then $J$ is integrable if and only if near each\/ $x\in X$ there exist local complex coordinates $(z_1,\ldots,z_n)$ defined on an open neighbourhood\/ $U$ of\/ $x$ in $X,$ such that\/ $z_j:U\ra\C$ is holomorphic for $j=1,\ldots,n$. In these coordinates, $J$ is locally standard. Writing $z_j=x_j+iy_j$, we have $TX\vert_U=\an{\frac{\pd}{\pd x_j},\frac{\pd}{\pd y_j}:1\le j\le n}_\R,$~with
\begin{equation*}
J\Bigl(\frac{\pd}{\pd x_j}\Bigr)=\frac{\pd}{\pd y_j},\qquad J\Bigl(\frac{\pd}{\pd y_j}\Bigr)=-\frac{\pd}{\pd x_j},\qquad j=1,\ldots,n.
\end{equation*}
\end{thm}

This paper will generalize these ideas from ordinary manifolds $X$, locally modelled on $\R^m$, to {\it manifolds with corners}, locally modelled on $[0,\iy)^k\t\R^{m-k}$, and to {\it manifolds with generalized corners\/} (or {\it manifolds with g-corners\/}) from the second author \cite{Joyc2}, which have a larger set of local models $X_P$ depending on a `weakly toric monoid' $P$, described in~\S\ref{bc22}. 

As in \S\ref{bc21}, a manifold with corners $X$ has two notions of tangent bundle, the ordinary tangent bundle $TX$, and the {\it b-tangent bundle\/} ${}^bTX$, introduced by Melrose \cite{Melr2,Melr3}. If $(x_1,\ldots,x_m)\in[0,\iy)^k\t\R^{m-k}$ are local coordinates on open $U\subset X$, so that $x_1,\ldots,x_k$ take values in $[0,\iy)$, then 
\begin{equation*}
TX\vert_U\!\cong\!\Bigl\langle \frac{\pd}{\pd x_1},\ldots,\frac{\pd}{\pd x_m}\Bigr\rangle_\R,\;\> 
{}^bTX\vert_U\!\cong\!\Bigl\langle x_1\frac{\pd}{\pd x_1},\ldots,x_k\frac{\pd}{\pd x_k},\frac{\pd}{\pd x_{k+1}},\ldots,\frac{\pd}{\pd x_m}\Bigr\rangle_\R.
\end{equation*}

Manifolds with g-corners $X$ should probably be considered singular from the point of view of ordinary differential geometry --- for example, the ordinary tangent bundle $TX$ may not exist as a vector bundle --- but they have well behaved b-tangent bundles ${}^bTX$. Here is part of Definition \ref{bc3def1} below.

\begin{dfn}
\label{bc1def2}
Let $X$ be a manifold with corners or g-corners, of real dimension $2n$. Then $X$ has a {\it depth stratification\/} $X=\coprod_{k=0}^{2n}S^k(X)$, where $S^k(X)$ is a manifold without boundary of dimension $2n-k$, and is roughly the points in $X$ which lie in the intersection of $k$ boundary hypersurfaces. There is a natural exact sequence of vector bundles on $S^k(X)$
\begin{equation*}
\xymatrix@C=20pt{ 0 \ar[r] & {}^bN_{S^k(X)} \ar[rr]^{{}^bi} && {}^bTX\vert_{S^k(X)}
\ar[rr]^{{}^b\pi} && T(S^k(X)) \ar[r] & 0, }
\end{equation*}
where ${}^bN_{S^k(X)}$ is the {\it b-normal bundle\/} of $S^k(X)$ in $X$, a rank $k$ vector bundle which has a natural flat connection~$\nabla^{{}^bN_{S^k(X)}}$.

A {\it b-almost complex structure\/} $J$ on $X$ is a morphism of vector bundles $J:{}^bTX\ra{}^bTX$ such that $J^2=-\id_{TX}$, and for all $0\le k\le 2n$ we have 
\begin{equation*}
{}^bi\bigl({}^bN_{S^k(X)}\bigr)\cap J\bigl[{}^bi\bigl({}^bN_{S^k(X)}\bigr)\bigr]=0
\end{equation*}
in subbundles of ${}^bTX\vert_{S^k(X)}$. This is only possible if $S^k(X)=\es$ when $n<k\le 2n$. It holds automatically if $k=0$ or~1.

As in Definition \ref{bc1def1}, a b-almost complex structure has a {\it Nijenhuis tensor\/} $N_J\in\Ga^\iy({}^bT^*X\ot{}^bT^*X\ot{}^bTX)$ such that \eq{bc1eq1} holds for all b-vector fields $v,w$ on $X$. We call $J$ {\it integrable\/} if $N_J=0$, and then we call $(X,J)$ a {\it b-complex manifold with (g-)corners}.
\end{dfn}

\subsection{Main results}

This paper will prove Theorems \ref{bc3thm2} and \ref{bc3thm3} below, to which readers are referred for detailed statements. Theorem \ref{bc3thm2} shows that if $(X,J)$ is b-complex manifold with (g-)corners then on each stratum $S^k(X)$ we can find local coordinates $(\th_1,\ldots,\th_k,z_{k+1},\ldots,z_n)$ in $\R^k\t\C^{n-k}$ which put $J\vert_{S^k(X)}:{}^bTX\vert_{S^k(X)}\ra {}^bTX\vert_{S^k(X)}$, and the natural Lie algebroid structure on ${}^bTX\vert_{S^k(X)}$, into a standard form; essentially, $z_{k+1},\ldots,z_n$ are holomorphic functions on $S^k(X)$. 

Theorem \ref{bc3thm3} is a formal analogue of Theorem \ref{bc1thm1} for b-complex manifolds with (g-)corners. For simplicity, we state here a result equivalent to Theorem \ref{bc3thm3} for manifolds with ordinary corners rather than g-corners:

\begin{thm}
\label{bc1thm2}
Let\/ $(X,J)$ be a b-complex manifold with corners, of real dimension $2n,$ and let\/ $x\in S^k(X)$ for\/ $0\le k\le n$. Then we can choose an open neighbourhood\/ $U$ of\/ $x$ in $X$ and global coordinates $(\la_1,\ldots,\la_k,\th_1,\ldots,\th_k,x_{k+1},\ab\ldots,\ab x_n,\ab y_{k+1},\ab\ldots,y_n)\in[0,\iy)^k\t\R^{2n-k}$ on $U,$ where $\la_1,\ldots,\la_k$ take values in $[0,\iy),$ with\/ $S^k(X)$ locally defined by $\la_1=\cdots=\la_k=0,$ and\/ $\la_1(x)=\cdots=\la_k(x)=0,$ such that defining the `standard' b-complex structure $J_\rst$ on $U$ by
\begin{equation*}
J_\rst\Bigl(\la_j\frac{\pd}{\pd\la_j}\Bigr)\!=\!\frac{\pd}{\pd\th_j},\;\>  J_\rst\Bigl(\frac{\pd}{\pd\th_j}\Bigr)\!=\!-\la_j\frac{\pd}{\pd\la_j},\;\>
J_\rst\Bigl(\frac{\pd}{\pd x_j}\Bigr)\!=\!\frac{\pd}{\pd y_j}, \;\>
J_\rst\Bigl(\frac{\pd}{\pd y_j}\Bigr)\!=\!-\frac{\pd}{\pd x_j}, 
\end{equation*}
then $J$ and\/ $J_\rst$ are equal on the formal completion of\/ $U$ along $S^k(X)\cap U$. That is, $J-J_\rst$ vanishes along $S^k(X)\cap U$ to infinite order in\/~$\la_1,\ldots,\la_k$.
\end{thm}

For manifolds with boundary, Theorem \ref{bc1thm2} was proved by Mendoza~\cite{Mend2}. 

In the standard b-complex structure $J_\rst$ we have holomorphic functions on~$U$:
\begin{equation*}
(\la_1e^{i\th_1},\ldots,\la_ke^{i\th_k},x_{k+1}+iy_{k+1},\ldots,x_n+iy_n).
\end{equation*}
All other $J_\rst$-holomorphic functions on $U$ may be written locally as holomorphic functions of these. However, they are not complex coordinates on $U$, even locally, as they are degenerate: when $\la_j=0$ we cannot recover $\th_j$ from~$\la_je^{i\th_j}$.

The sequel \cite{ArJo2} will apply b-complex manifolds with g-corners in log geometry. Let $S$ be a log $\C$-scheme, or log complex analytic space. Kato and Nakayama \cite{KaNa1} explained how to functorially associate a topological space $S^\KN$ to $S$, the {\it Kato--Nakayama space}. This has important applications, e.g.\ in Mirror Symmetry. 

In \cite{ArJo2} we will show that if $S$ is log smooth and separated then $S^\KN$ can be given the structure of a b-complex manifold with g-corners, which is a b-complex manifold with corners if $S$ is also smooth. Conversely, if $X$ is a b-complex manifold with g-corners, we will give necessary and sufficient conditions for $X$ to be isomorphic to $S^\KN$, for $S$ a log smooth log complex analytic space that we can reconstruct from~$X$.
\medskip

\noindent{\it Acknowledgements.} The authors would like to thank Pierrick Bousseau for useful conversations. For the purpose of open access, the authors have applied a CC BY public copyright licence to any Author Accepted Manuscript (AAM) version arising from this submission. 

\section{Manifolds with (generalized) corners}
\label{bc2}

{\it Manifolds with corners\/} are analogues of smooth manifolds which are locally modelled on $\R^m_k=[0,\iy)^k\t\R^{m-k}$ for $0\le k\le m$. They include ordinary manifolds, locally modelled on $\R^m_0=\R^m$, and {\it manifolds with boundary}, locally modelled on $\R^m_1=[0,\iy)\t\R^{m-1}$. They were introduced by Cerf \cite{Cerf} and Douady \cite{Doua} in the early 1960s, and studied by many people including Melrose \cite{Melr2,Melr3}, \cite[\S 1]{KoMe} and the second author \cite{Joyc1,Joyc2}. There are several different notions of smooth map between manifolds with corners in the literature; our definition in \S\ref{bc21} follows Melrose \cite{Melr2,Melr3}, who calls them {\it b-maps}, and the second author \cite{Joyc2}. Manifolds with corners and smooth maps form a category~$\Manc$.

We will be interested in an extension introduced by the second author \cite{Joyc2}, called {\it manifolds with generalized corners}, or {\it manifolds with g-corners}, explained in \S\ref{bc22}. These were inspired by the `interior binomial varieties' of Kottke--Melrose \cite[\S 9]{KoMe}. They form a category $\Mangc$ containing $\Manc$ as a full subcategory. The local models for manifolds with g-corners are spaces $X_P$ depending on a {\it weakly toric monoid\/} $P$, where $X_P\cong\R^m_k$ when~$P\cong\N^k\t\Z^{m-k}$. 

We will discuss only topics that we need later. In particular, we will not define the boundary $\pd X$ or corners $C(X)=\coprod_{k=0}^{\dim X}C_k(X)$ of a manifold with corners $X$, though these are very important.

\subsection{Manifolds with corners}
\label{bc21}

See Melrose \cite{Melr2,Melr3}, \cite[\S 1]{KoMe} and the second author \cite{Joyc1}, \cite[\S 2]{Joyc2} for the material of this section.

\subsubsection{The definition of manifolds with corners}
\label{bc211}

\begin{dfn}
\label{bc2def1}
Use the notation $\R^m_k=[0,\iy)^k\t\R^{m-k}$ for $0\le k\le m$, and write points of $\R^m_k$ as $u=(x_1,\ldots,x_m)$ for $x_1,\ldots,x_k\in[0,\iy)$, $x_{k+1},\ldots,x_m\in\R$. Let $U\subseteq\R^m_k$ and $V\subseteq \R^n_l$ be open, and $f=(f_1,\ldots,f_n):U\ra V$ be a continuous map, so that $f_j=f_j(x_1,\ldots,x_m)$ maps $U\ra[0,\iy)$ for $j=1,\ldots,l$ and $U\ra\R$ for $j=l+1,\ldots,n$. Then we say:
\begin{itemize}
\setlength{\itemsep}{0pt}
\setlength{\parsep}{0pt}
\item[(a)] $f$ is {\it weakly smooth\/} if all derivatives $\frac{\pd^{a_1+\cdots+a_m}}{\pd x_1^{a_1}\cdots\pd x_m^{a_m}}f_j(x_1,\ldots,x_m):U\ra\R$ exist and are continuous for all $1\le j\le n$ and $a_1,\ldots,a_m\ge 0$, including one-sided derivatives where $x_i=0$ for~$1\le i\le k$.
\item[(b)] $f$ is {\it smooth\/} if it is weakly smooth and every $u=(x_1,\ldots,x_m)\in U$ has an open neighbourhood $\ti U$ in $U$ such that for each $j=1,\ldots,l$, either:
\begin{itemize}
\setlength{\itemsep}{0pt}
\setlength{\parsep}{0pt}
\item[(i)] we may uniquely write $f_j(\ti x_1,\ldots,\ti x_m)=F_j(\ti x_1,\ldots,\ti x_m)\cdot\ti x_1^{a_{1,j}}\cdots\ti x_k^{a_{k,j}}$ for all $(\ti x_1,\ldots,\ti x_m)$ in $\ti U$, where $F_j:\ti U\ra(0,\iy)$ is weakly smooth and $a_{1,j},\ldots,a_{k,j}\in\N=\{0,1,2,\ldots\}$, with $a_{i,j}=0$ if $x_i\ne 0$; or 
\item[(ii)] $f_j\vert_{\smash{\ti U}}=0$.
\end{itemize}
\item[(c)] $f$ is {\it interior\/} if it is smooth, and case (b)(ii) does not occur.
\item[(d)] $f$ is a {\it diffeomorphism\/} if it is a smooth bijection with smooth inverse.
\end{itemize}
\end{dfn}

\begin{dfn} 
\label{bc2def2}
Let $X$ be a second countable Hausdorff topological space. An {\it $m$-dimensional chart on\/} $X$ is a pair $(U,\phi)$, where $U\subseteq\R^m_k$ is open for some $0\le k\le m$, and $\phi:U\ra X$ is a homeomorphism with an open set~$\phi(U)\subseteq X$. Let $(U,\phi),(V,\psi)$ be $m$-dimensional charts on $X$. We call $(U,\phi)$ and $(V,\psi)$ {\it compatible\/} if $\psi^{-1}\ci\phi:\phi^{-1}\bigl(\phi(U)\cap\psi(V)\bigr)\ra\psi^{-1}\bigl(\phi(U)\cap\psi(V)\bigr)$ is a diffeomorphism between open subsets of $\R^m_k,\R^m_l$, in the sense of Definition~\ref{bc2def1}(d).

An $m$-{\it dimensional atlas\/} for $X$ is a system $\{(U_a,\phi_a):a\in A\}$ of pairwise compatible $m$-dimensional charts on $X$ with $X=\bigcup_{a\in A}\phi_a(U_a)$. We call such an atlas {\it maximal\/} if it is not a proper subset of any other atlas. Any atlas $\{(U_a,\phi_a):a\in A\}$ is contained in a unique maximal atlas, the set of all charts $(U,\phi)$ of this type on $X$ which are compatible with $(U_a,\phi_a)$ for all~$a\in A$.

An $m$-{\it dimensional manifold with corners\/} is a second countable Hausdorff topological space $X$ equipped with a maximal $m$-dimensional atlas. Usually we refer to $X$ as the manifold, leaving the atlas implicit, and by a {\it chart\/ $(U,\phi)$ on\/} $X$, we mean an element of the maximal atlas.

Now let $X,Y$ be manifolds with corners of dimensions $m,n$, and $f:X\ra Y$ a continuous map. We call $f$ {\it weakly smooth}, or {\it smooth}, or {\it interior}, if whenever $(U,\phi),(V,\psi)$ are charts on $X,Y$ with $U\subseteq\R^m_k$, $V\subseteq\R^n_l$ open, then $\psi^{-1}\ci f\ci\phi:(f\ci\phi)^{-1}(\psi(V))\ra V$ is weakly smooth, or smooth, or interior, respectively, as maps between open subsets of $\R^m_k,\R^n_l$ in the sense of Definition~\ref{bc2def1}.

We write $\Manc$ for the category with objects manifolds with corners $X,Y,$ and morphisms smooth maps $f:X\ra Y$ in the sense above. We will also write $\Mancin$ for the subcategory of $\Manc$ with morphisms interior maps.

We call a manifold with corners $X$ a {\it manifold with boundary\/} if it can be covered by charts $(U,\phi)$ with $U\subseteq\R^m_k$ open for $k=0$ or 1, and a {\it manifold without boundary\/} if it can be covered by $(U,\phi)$ with $U\subseteq\R^m_k$ open for $k=0$ only. Manifolds without boundary are ordinary manifolds. Write $\Man\subset\Manb\subset\Manc$ for the full subcategories of manifolds without, and with, boundary.
\end{dfn}

\subsubsection{(B-)tangent bundles and (b-)cotangent bundles}
\label{bc212}

Manifolds with corners $X$ have two notions of tangent bundle with functorial properties, the ({\it ordinary\/}) {\it tangent bundle\/} $TX\ra X$, the obvious generalization of tangent bundles of manifolds without boundary, and the {\it b-tangent bundle\/} ${}^bTX\ra X$ introduced by Melrose \cite[\S 2.2]{Melr2}, \cite[\S I.10]{Melr3}. Taking duals gives two notions of cotangent bundle~$T^*X,{}^bT^*X$. 

\begin{dfn}
\label{bc2def3}
Let $X$ be an $m$-manifold with corners. We can define {\it vector bundles\/} on $X$ in the usual way. Then $X$ has four natural vector bundles, the {\it tangent bundle\/} $TX\ra X$, its dual the {\it cotangent bundle\/} $T^*X\ra X$, the {\it b-tangent bundle\/} ${}^bTX\ra X$, and its dual the {\it b-cotangent bundle\/} ${}^bT^*X\ra X$. All have rank $m$. The simplest way to define them is in local coordinates. Let $(x_1,\ldots,x_m)\in\R^m_k$ be local coordinates on an open subset $U\subset X$, so that $x_1,\ldots,x_k$ take values in $[0,\iy)$ and $x_{k+1},\ldots,x_m$ take values in $\R$. Then $TX,T^*X,{}^bTX,{}^bT^*X$ are trivial vector bundles on $U$, with bases of sections
\e
\begin{aligned}
TX\vert_U&\cong\Bigl\langle \frac{\pd}{\pd x_1},\ldots,\frac{\pd}{\pd x_m}\Bigr\rangle_\R,\qquad T^*X\vert_U\cong\ban{\d x_1,\ldots,\d x_m}_\R,\\
{}^bTX\vert_U&\cong\Bigl\langle x_1\frac{\pd}{\pd x_1},\ldots,x_k\frac{\pd}{\pd x_k},\frac{\pd}{\pd x_{k+1}},\ldots,\frac{\pd}{\pd x_m}\Bigr\rangle_\R,\\
{}^bT^*X\vert_U&\cong\ban{x^{-1}\d x_1,\ldots,x^{-1}\d x_k,\d x_{k+1},\ldots,\d x_m}_\R.
\end{aligned}
\label{bc2eq1}
\e

Here $x_i\frac{\pd}{\pd x_i}$ and $x_i^{-1}\d x_i$ for $i=1,\ldots,k$ are formal symbols, which are well defined and nonzero when $x_i=0$. The trivializations \eq{bc2eq1} transform as the notation suggests under change of coordinates.

For $E\ra X$ a vector bundle, we write $\Ga^\iy(E)$ for the vector space of smooth sections of $E$. Elements of $\Ga^\iy(TX)$ and $\Ga^\iy({}^bTX)$ are called {\it vector fields\/} and {\it b-vector fields}, respectively. 

There is a natural vector bundle morphism $I_X:{}^bTX\ra TX$ mapping $x_i\frac{\pd}{\pd x_i}\mapsto x_i\cdot\frac{\pd}{\pd x_i}$ for $1\le i\le k$ and $\frac{\pd}{\pd x_i}\mapsto\frac{\pd}{\pd x_i}$ for $k<i\le m$. This induces an injective morphism $(I_X)_*:\Ga^\iy({}^bTX)\ra\Ga^\iy(TX)$. We can think of $\Ga^\iy({}^bTX)$ as the vector subspace of vector fields $v\in\Ga^\iy(TX)$ which are tangent to every boundary stratum of $X$.

Write $C^\iy(X)$ for the $\R$-algebra of smooth functions $f:X\ra\R$. Then vector fields and b-vector fields $v$ on $X$ both act as derivations $v:f\mapsto v(f)$. In coordinates $(x_1,\ldots,x_m)$ with trivialization \eq{bc2eq1}, the action is the obvious one, with $\frac{\pd}{\pd x_i}:f\mapsto\frac{\pd f}{\pd x_i}$ and $x_i\frac{\pd}{\pd x_i}:f\mapsto x_i\frac{\pd f}{\pd x_i}$. There are {\it Lie brackets\/} $[\,,\,]$ on both $\Ga^\iy(TX)$ and $\Ga^\iy({}^bTX)$, such that $[v,w]$ acts on $f\in C^\iy(X)$ as~$[v,w](f)=v(w(f))-w(v(f))$.

There is a {\it de Rham differential\/} $\d:C^\iy(X)\ra\Ga^\iy(T^*X)$ and {\it b-differential\/} ${}^b\d:C^\iy(X)\ra\Ga^\iy({}^bT^*X)$. These act in the obvious way in the trivializations \eq{bc2eq1}, so in particular ${}^b\d$ acts by
\begin{align*}
{}^b\d:f\longmapsto &\sum_{i=1}^k\Bigl(x_i\frac{\pd f}{\pd x_i}\Bigr)x_i^{-1}\d x_i+\sum_{i=k+1}^m\Bigl(\frac{\pd f}{\pd x_i}\Bigr)\d x_i.
\end{align*}

Let $f:X\ra Y$ be a smooth map of manifolds with corners. There are natural vector bundle morphisms $Tf:TX\ra f^*(TY)$ and $T^*f:f^*(T^*Y)\ra T^*X$, the {\it derivatives\/} of $f$, with the obvious functorial properties. If $f$ is an {\it interior\/} map we have morphisms 
${}^bTf:{}^bTX\ra f^*({}^bTY)$ and ${}^bT^*f:f^*({}^bT^*Y)\ra {}^bT^*X$, the {\it b-derivatives\/} of $f$. But these are not defined for non-interior $f$.
\end{dfn}

\subsubsection{The depth stratification, and (b-)normal bundles}
\label{bc213}

\begin{dfn}
\label{bc2def4}
Let $U\subseteq\R^m_k$ be open. For each $u=(u_1,\ldots,u_m)$ in $U$, define the {\it depth\/} $\depth_Uu$ of $u$ in $U$ to be the number of $u_1,\ldots,u_k$ which are zero.

Let $X$ be an $m$-manifold with corners. For $x\in X$, choose a chart $(U,\phi)$ on the manifold $X$ with $\phi(u)=x$ for $u\in U$, and define the {\it depth\/} $\depth_Xx$ of $x$ in $X$ by $\depth_Xx=\depth_Uu$. This is independent of the choice of $(U,\phi)$. For each $k=0,\ldots,m$, define the {\it depth\/ $k$ stratum\/} of $X$ to be
\begin{equation*}
S^k(X)=\bigl\{x\in X:\depth_Xx=k\bigr\}.
\end{equation*}
Then $X=\coprod_{k=0}^mS^k(X)$ and $\overline{S^k(X)}=\bigcup_{l=k}^mS^l(X)$. The {\it interior\/} of $X$ is $X^\ci=S^0(X)$. Each $S^k(X)$ has the unique structure of an $(m-k)$-manifold without boundary, such that if $x\in S^k(X)\subseteq X$ and $(x_1,\ldots,x_m)\in\R^m_k$ are local coordinates on $X$ near $x$ with $x=(0,\ldots,0,x_{k+1},\ldots,x_m)$, then $(x_{k+1},\ldots,x_m)\in\R^{m-k}$ are local coordinates on~$S^k(X)$.
\end{dfn}

The next definition is explained in \cite[\S 2.4]{Joyc2} using the $k$-{\it corners\/} $C_k(X)$, which we do not define, but whose interior is $C_k(X)^\ci\cong S^k(X)$. Our version is obtained by restricting from $C_k(X)$ to~$S^k(X)$.

\begin{dfn} 
\label{bc2def5}
Let $X$ be an $m$-manifold with corners, and $k=0,\ldots,m$. Then on the depth $k$ stratum $S^k(X)$ we have an exact sequence of vector bundles
\e
\xymatrix@C=15pt{ 0 \ar[r] & T(S^k(X)) \ar[rr]^i && TX\vert_{S^k(X)}
\ar[rr]^\pi && N_{S^k(X)} \ar[r] & 0, }
\label{bc2eq2}
\e
where $N_{S^k(X)}$ is the {\it normal bundle\/} of $S^k(X)$ in $X$, of rank $k$. In local coordinates $(x_1,\ldots,x_m)\in\R^m_k$ on $X$ near a point of $S^k(X)$, so that $S^k(X)$ is locally defined by $x_1=\cdots=x_k=0$ and $(x_{k+1},\ldots,x_m)$ are local coordinates on $S^k(X)$, equation \eq{bc2eq2} becomes 
\begin{equation*}
\xymatrix@C=15pt{ 0 \ar[r] & \bigl\langle \frac{\pd}{\pd x_{k+1}},\ldots,\frac{\pd}{\pd x_m}\bigr\rangle_\R \ar[r] & \bigl\langle \frac{\pd}{\pd x_1},\ldots,\frac{\pd}{\pd x_m}\bigr\rangle_\R
\ar[r] & \bigl\langle \frac{\pd}{\pd x_1},\ldots,\frac{\pd}{\pd x_k}\bigr\rangle_\R \ar[r] & 0. }
\end{equation*}

Similarly, for the b-tangent bundle we also have an exact sequence on~$S^k(X)$:
\e
\xymatrix@C=20pt{ 0 \ar[r] & {}^bN_{S^k(X)} \ar[rr]^{{}^bi} && {}^bTX\vert_{S^k(X)}
\ar[rr]^{{}^b\pi} && T(S^k(X)) \ar[r] & 0, }
\label{bc2eq3}
\e
where ${}^bN_{S^k(X)}$ is the {\it b-normal bundle\/} of $S^k(X)$ in $X$. Note that the order in \eq{bc2eq3} is reversed compared to \eq{bc2eq2}. In local coordinates $(x_1,\ldots,x_m)\in\R^m_k$ on $X$ near a point of $S^k(X)$, equation \eq{bc2eq3} becomes
\e
\text{\begin{footnotesize}$\displaystyle
\xymatrix@C=12pt{ 0 \ar[r] & \bigl\langle x_1\frac{\pd}{\pd x_1},\ldots,x_k\frac{\pd}{\pd x_k}\bigr\rangle_\R\! \ar[r] & {\begin{subarray}{l} \ts\bigl\langle x_1\frac{\pd}{\pd x_1},\ldots,x_k\frac{\pd}{\pd x_k}, \\ \ts \;\frac{\pd}{\pd x_{k+1}},\ldots,\frac{\pd}{\pd x_m}\bigr\rangle_\R \end{subarray}}
\ar[r] & \bigl\langle \frac{\pd}{\pd x_{k+1}},\ldots,\frac{\pd}{\pd x_m}\bigr\rangle_\R\! \ar[r] & 0. }$\end{footnotesize}}
\label{bc2eq4}
\e

As in \cite[\S 2.4]{Joyc2}, the b-normal bundle ${}^bN_{S^k(X)}\ra S^k(X)$ has an interesting property: it has a {\it natural flat connection\/} $\nabla^{{}^bN_{S^k(X)}}$, such that $x_1\frac{\pd}{\pd x_1},\ldots,x_k\frac{\pd}{\pd x_k}$ in \eq{bc2eq4} are local constant sections. Define the {\it monoid bundle\/} $M_{S^k(X)}\subset{}^bN_{S^k(X)}$ to be the set of points which in local coordinates as in \eq{bc2eq4} are of the form $a_1(x_1\frac{\pd}{\pd x_1})+\ldots+a_k(x_k\frac{\pd}{\pd x_k})$ for $a_1,\ldots,a_k\in\N=\{0,1,2,\ldots\}$. Then $M_{S^k(X)}\ra X$ is a bundle with fibre $\N^k$, which we think of as a {\it local system of monoids\/} under addition in $\N^k$. We have ${}^bN_{S^k(X)}\cong M_{S^k(X)}\ot_\N\R$, and the flat connection on ${}^bN_{S^k(X)}$ is induced from the local system structure on~$M_{S^k(X)}$.
\end{dfn}

\subsubsection{A natural Lie algebroid on $S^k(X)$}
\label{bc214}

Originally introduced by Pradines \cite{Prad}, Lie algebroids over manifolds provide a common generalization of tangent bundles and bundles of Lie algebras. For a comprehensive study and further references see MacKenzie~\cite{Mack2}. 

\begin{dfn}
\label{bc2def6}
A {\it Lie algebroid\/} $\pi\colon A \ra X$ over a manifold $X$ is a smooth vector bundle $A$ over $X$ endowed with the data of a Lie bracket $[\,,\,]_A$ on its space of smooth sections $\Ga^\iy(A)$, and a map of vector bundles $\rho \colon A \ra TX $, called the {\it anchor map}, such that the following conditions are satisfied for all $s,s' \in \Ga^\iy(A)$ and~$f\in C^\iy(X)$:
\begin{itemize}
\setlength{\itemsep}{0pt}
\setlength{\parsep}{0pt}
\item[(i)] $[\rho(s),\rho(s')]=\rho([s,s']_A)$,
\item[(ii)] $[s,fs']_A=(\rho(s)(f))s' + f[s,s']_A$.  
\end{itemize}
\end{dfn}

\begin{dfn}
\label{bc2def7}
Let $X$ be an $m$-manifold with corners, and $k=0,\ldots,m$. Then there is a natural Lie algebroid on the depth $k$ stratum $S^k(X)$, defined as follows. We take the vector bundle $A$ to be ${}^bTX\vert_{S^k(X)}$, and the anchor map $\rho$ to be ${}^b\pi:{}^bTX\vert_{S^k(X)}\ra T(S^k(X))$ as in \eq{bc2eq3}. If $u,v\in \Ga^\iy({}^bTX\vert_{S^k(X)})$, we choose smooth extensions $\ti u,\ti v$ of $u,v$ to an open neighbourhood $V$ of $S^k(X)$ in $X$ with $\ti u\vert_{S^k(X)}=u$ and $\ti v\vert_{S^k(X)}=v$, and we define the Lie bracket $[\,,\,]_{S^k}$ on $\Ga^\iy({}^bTX\vert_{S^k(X)})$ by $[u,v]_{S^k}=[\ti u,\ti v]\vert_{S^k(X)}$, using the Lie bracket on $\Ga^\iy({}^bTV)$ from \S\ref{bc212}. One can show this is well defined. 

In the local coordinate trivialization \eq{bc2eq4} of ${}^bTX\vert_{S^k(X)}$, we have
\ea
&\Bigl[u_1\cdot x_1\frac{\pd}{\pd x_1}+\cdots+u_k\cdot x_k\frac{\pd}{\pd x_k}+u_{k+1}\frac{\pd}{\pd x_{k+1}}+\cdots+u_m\frac{\pd}{\pd x_m},
\label{bc2eq5}\\
&\quad v_1\cdot x_1\frac{\pd}{\pd x_1}+\cdots+v_k\cdot x_k\frac{\pd}{\pd x_k}+v_{k+1}\frac{\pd}{\pd x_{k+1}}+\cdots+v_m\frac{\pd}{\pd x_m}\Bigr]_{S^k}\nonumber\\
&=\sum_{\begin{subarray}{l} k<i\le m, \\
1\le j\le k \end{subarray}}\Bigl(u_i\frac{\pd v_j}{\pd x_i}-v_i\frac{\pd u_j}{\pd x_i}\Bigr)x_j\frac{\pd}{\pd x_j}+\sum_{k<i,j\le m}\Bigl(u_i\frac{\pd v_j}{\pd x_i}-v_i\frac{\pd u_j}{\pd x_i}\Bigr)\frac{\pd}{\pd x_j},
\nonumber
\ea
for $u_i,v_j$ local functions of the coordinates $(x_{k+1},\ldots,x_m)$ on $S^k(X)$.

The flat connection $\nabla^{{}^bN_{S^k(X)}}$ on ${}^bN_{S^k(X)}$ may be characterized using the Lie algebroid as follows: a local section $u$ of ${}^bN_{S^k(X)}\subset {}^bTX\vert_{S^k(X)}$ is constant under $\nabla^{{}^bN_{S^k(X)}}$ if and only if $[u,v]_{S^k}=0$ for all local sections $v$ of~${}^bTX\vert_{S^k(X)}$.
\end{dfn}

\subsection{Manifolds with g-corners}
\label{bc22}

Manifolds with g-corners were introduced by the second author \cite{Joyc2}, building on ideas in Kottke--Melrose \cite[\S 9]{KoMe}. Much of the theory of manifolds with corners --- in particular, b-(co)tangent bundles ${}^bTX,{}^bT^*X\ra X$ --- extends nicely to manifolds with g-corners. Manifolds with g-corners behave better in some ways, e.g.\ transverse fibre products exist in the category $\Mangc$ of manifolds with g-corners \cite[\S 4.3]{Joyc2}, but they exist in $\Manc$ only under restrictive conditions.

Manifolds with g-corners arise naturally in some problems. For example, Ma'u and Woodward \cite{MaWo} define moduli spaces $\oM_{n,1}$ of `stable $n$-marked quilted discs'. As in \cite[\S 6]{MaWo}, for $n\ge 4$ these are not manifolds with corners, but have an exotic corner structure; in the language of \cite{Joyc2}, the $\oM_{n,1}$ are manifolds with g-corners. We will show in the sequel \cite{ArJo2} that if $S$ is a log smooth log $\C$-scheme or complex analytic space then the `Kato--Nakayama space' $S^\KN$ is a b-complex manifold with g-corners, but it may not be a manifold with corners.

\subsubsection{Monoids}
\label{bc221}

\begin{dfn}
\label{bc2def8}
A {\it monoid\/} is a set $P$ with an associative binary operation, often denoted additively by $+$, with an identity element $0$. Throughout this paper, all monoids are assumed to be commutative. An element of $P$ is called a {\it unit\/} if it has an inverse. The set of units of $P$, which itself is a submonoid of $P$, is denoted by $P^\t$. The {\it Grothendieck group\/} $P^\gp$ of a monoid $P$ is an abelian group with a monoid morphism $\iota \colon P \ra P^\gp$, with the universal property that for any morphism $f:P\ra A$ from $P$ to an abelian group $A$, there is a unique group morphism $g:P^\gp \ra A$ such that $f=g\ci \iota$. A monoid $P$ is called:
\begin{itemize}
\setlength{\itemsep}{0pt}
\setlength{\parsep}{0pt}
\item[(i)] {\it integral\/} if the cancellation property holds, or equivalently if $P \ra P^\gp$ is injective.
\item[(ii)] {\it fine\/} if it is finitely generated and integral.
\item[(iii)] {\it saturated\/} if it is integral and $P=\{m\in P^\gp:nm\in P\subseteq P^\gp$ for some $n\in\Z_{> 0} \}$.
\item[(iv)] {\it sharp} if $P^\t = \{ 0 \}$,
\item[(v)] {\it torsion-free} if it is integral and $P^\gp$ is torsion-free,
\item[(vi)] {\it weakly toric\/} if it is fine, saturated and torsion free.
\item[(vii)] {\it toric\/} if it is weakly toric and sharp. 
\end{itemize}
Every weakly toric monoid is isomorphic to a submonoid of some~$\Z^k$.

Some authors, such as Ogus \cite{Ogus}, call our weakly toric monoids `toric'. 

If $P$ is a weakly toric monoid then $P\cong Q\t\Z^n$ where $Q$ is a toric monoid. We have $\Z^n=P^\t$ and $Q\cong P/P^\t$, so the exact sequence $0\ra\Z^n\ra P\ra Q\ra 0$ is natural. If $P$ is a monoid, the {\it dual monoid\/} is $P^\vee=\Hom(P,\N)$. There is a natural morphism $\eta(P):P\ra(P^\vee)^\vee$ given by $\eta:p\mapsto\bigl(\mu\mapsto\mu(p)\bigr)$ for $p\in P$ and $\mu\in P^\vee$. This is an isomorphism if $P$ is toric.
\end{dfn}

The next definition is taken from Ogus \cite[\S 1.4]{Ogus}.

\begin{dfn}
\label{bc2def9}
An {\it ideal\/} $I$ of a monoid $P$ is a subset $I\subsetneq P$ such that for all $i\in I$ and $p\in P$ we have $p+i\in I$. An ideal $I$ is called {\it prime\/} if $p,q\in P$ and $p+q\in I$ imply that $p\in I$ or~$q\in I$. 

A submonoid $F\subseteq P$ is called a {\it face\/} of $P$ if $p,q\in P$ and $p+q\in F$ imply that $p\in F$ and $q\in F$. If is easy to see that $F\subseteq P$ is a face of $P$ if and only if $I=P\sm F$ is a prime ideal in $P$. This gives a bijection $F\longleftrightarrow I=P\sm F$ between faces $F$ of $P$ and prime ideals $I$ in~$P$.

The {\it codimension\/} $\codim F$ of a face $F\subseteq P$ is the rank of the abelian group $(P/F)^\gp$, which is defined when $(P/F)^\gp$ is finitely generated. If $P$ is weakly toric then~$\rank F+\codim F=\rank P$.
\end{dfn}

\subsubsection{The model spaces $X_P,$ for $P$ a weakly toric monoid}
\label{bc222}

As in \S\ref{bc21}, manifolds with corners are locally modelled on $\R^m_k=[0,\iy)^k\t\R^{m-k}$ for $0\le k\le m$. We will define manifolds with g-corners in \S\ref{bc223} to be locally modelled on spaces $X_P$ depending on a weakly toric monoid $P$. First we define these spaces $X_P$, and `smooth maps' between them.

\begin{dfn} 
\label{bc2def10}
Let $P$ be a weakly toric monoid. Define $X_P$ to be the set of monoid morphisms $x:P\ra[0,\iy)$, where $([0,\iy),\cdot)$ is the monoid $[0,\iy)$ with operation multiplication and identity 1. Define the {\it interior\/} $X_P^\ci\subset X_P$ of $X_P$ to be the subset of $x$ with~$x(P)\subseteq(0,\iy)\subset[0,\iy)$. For each $p\in P$, define a function $\la_p:X_P\ra[0,\iy)$ by $\la_p(x)=x(p)$. Then $\la_{p+q}=\la_p\cdot\la_q$ for $p,q\in P$, and~$\la_0=1$.

Define a topology on $X_P$ to be the weakest topology such that $\la_p:X_P\ra[0,\iy)$ is continuous for all $p\in P$. This makes $X_P$ into a locally compact, Hausdorff topological space, and $X_P^\ci$ is open in $X_P$. If $U\subseteq X_P$ is an open set, define the {\it interior\/} $U^\ci$ of $U$ to be~$U^\ci=U\cap X_P^\ci$.

Note that $X_P$ and $U$ are not manifolds, in general, so smooth functions on $X_P,U$ are not yet defined. Let $f:U\ra\R$ be a continuous function. We say that $f$ is a {\it smooth function\/} $U\ra\R$ if there exist $r_1,\ldots,r_n\in P$, an open subset $W\subseteq[0,\iy)^n$, and a smooth map $g:W\ra\R$ (in the usual sense, as in \S\ref{bc21}), such that for all $x\in U$ we have $(x(r_1),\ldots,x(r_n))\in W$ and
\e
f(x)=g\bigl(x(r_1),\ldots,x(r_n)\bigr)=g\bigl(\la_{r_1}(x),\ldots,\la_{r_n}(x)\bigr).
\label{bc2eq6}
\e
We say that a continuous function $f:U\ra(0,\iy)$ is {\it smooth\/} if $f$ is smooth as a map~$U\ra\R$.

We say that a continuous function $f:U\ra[0,\iy)$ is {\it smooth\/} if on each connected component $U'$ of $U$, we either have $f\vert_{U'}=\la_p\vert_{U'}\cdot h$, where $p\in P$ and $h:U'\ra(0,\iy)$ is smooth, or $f\vert_{U'}=0$. Note that $f$ is smooth as a map $U\ra[0,\iy)$ implies that $f$ is smooth as a map $f:U\ra\R$, but not vice versa. 

Now let $Q$ be another weakly toric monoid, and $V\subseteq X_Q$ an open set. We say that a continuous map $f:U\ra V$ is {\it smooth\/} if $\la_q\ci f:U\ra[0,\iy)$ is smooth for all $q\in Q$. We say that $f$ is a {\it diffeomorphism\/} if $f$ is invertible and $f,f^{-1}$ are smooth. We say that $f$ is {\it interior\/} if $f$ is smooth and~$f(U^\ci)\subseteq V^\ci$. 

The identity map $\id_U:U\ra U$ is smooth and interior. If $P,Q$ are weakly toric monoids then so is $P\t Q$, and there is an isomorphism~$X_{P\t Q}\cong X_P\t X_Q$.

If $P$ is a {\it toric\/} monoid, define $\de_0:P\ra[0,\iy)$ by $\de_0(P)=1$ if $p=0$ and $\de_0(p)=0$ otherwise. Then $\de_0$ is a monoid morphism, since $P^\t=\{0\}$ as $P$ is sharp, so $\de_0\in X_P$. We call $\de_0$ the {\it vertex\/} of $X_P$.
\end{dfn}

Here is \cite[Prop.~3.14(a),(b)]{Joyc2}.

\begin{prop}
\label{bc2prop1}
Suppose\/ $P$ is a weakly toric monoid. Choose generators $p_1,\ldots,p_m$ for $P,$ and a generating set of relations for\/ $p_1,\ldots,p_m$ of the form
\begin{equation*}
a_1^jp_1+\cdots+a_m^jp_m=b_1^jp_1+\cdots+b_m^jp_m\quad\text{in $P$ for $j=1,\ldots,k,$}
\end{equation*}
where\/ $a_i^j,b_i^j\in\N$ for\/ $i=1,\ldots,m$ and\/ $j=1,\ldots,k$. Then:
\begin{itemize}
\setlength{\itemsep}{0pt}
\setlength{\parsep}{0pt}
\item[{\bf(a)}] $\la_{p_1}\t\cdots\t\la_{p_m}:X_P\ra[0,\iy)^m$ is a homeomorphism from $X_P$ to
\e
\!\!\!\!\!\!\!\!\!\! X_P'\!=\!\bigl\{(x_1,\ldots,x_m)\!\in\![0,\iy)^m:x_1^{a_1^j}\cdots x_m^{a_m^j}\!=\!x_1^{b_1^j}\cdots x_m^{b_m^j},\; j\!=\!1,\ldots,k\bigr\},
\label{bc2eq7}
\e
regarding $X_P'$ as a closed subset of\/ $[0,\iy)^m$ with the induced topology.
\item[{\bf(b)}] Let\/ $U\subseteq X_P$ be open, and write\/ $U'=(\la_{p_1}\t\cdots\t\la_{p_m})(U)$ for the corresponding open subset of\/ $X_P'$. Then $f:U\ra\R$ is smooth in the sense of Definition\/ {\rm\ref{bc2def10}} if and only if there exists an open neighbourhood\/ $W$ of\/ $U'$ in $[0,\iy)^m$ and a smooth map $g:W\ra\R$ in the sense of\/ {\rm\S\ref{bc21},} regarding $W$ as a manifold with corners, such that\/ $f=g\ci(\la_{p_1}\t\cdots\t\la_{p_m}):U\ra\R$. The analogues hold for smooth maps $U\ra[0,\iy)$ and\/~$W\ra[0,\iy)$.
\end{itemize}
\end{prop}

\begin{ex}
\label{bc2ex1}
{\bf(i)} When $P=\N$, points of $X_\N$ are monoid morphisms $x:\N\ra([0,\iy),\cdot)$, which may be written uniquely in the form $x(n)=y^n$, $n\in\N$, for $y\in[0,\iy)$. This gives an identification $X_\N\cong[0,\iy)$ mapping $x\mapsto y=x(1)$. In Proposition \ref{bc2prop1}, we may take $P=\N$ to be generated by $p_1=1$, with no relations. Then part (a) shows that $\la_1:X_\N\ra X_\N'=[0,\iy)$ is a homeomorphism, the same identification $X_\N\cong[0,\iy)$ as above.
\smallskip

\noindent{\bf(ii)} When $P=\Z$, points of $X_\Z$ are monoid morphisms $x:\Z\ra\bigl([0,\iy),\cdot\bigr)$, which may be written uniquely in the form $x(n)=e^{ny}$ for $y\in\R$. This gives an identification $X_\Z\cong\R$ mapping $x\mapsto y=\log x(1)$. In Proposition \ref{bc2prop1}, we may take $P=\Z$ to be generated by $p_1=1$ and $p_2=-1$, with one relation $p_1+p_2=0$. Then part (a) shows $\la_1\t\la_{-1}$ is a homeomorphism from $X_\Z$ to
\begin{equation*}
X_\Z'=\bigl\{(x_1,x_2)\in[0,\iy)^2:x_1x_2=1\bigr\}.
\end{equation*}
In terms of the identification $X_\Z\cong\R\in y$ above, we have
\begin{equation*}
X_\Z'=\bigl\{(e^y,e^{-y}):y\in\R\bigr\}\cong\R.
\end{equation*}

\noindent{\bf(iii)} When $P=\N^k\t\Z^{m-k}$, combining {\bf(i)\rm,\bf(ii)}, points of $X_P$ are monoid morphisms $x:P\ra([0,\iy),\cdot)$, which may be written uniquely in the form
\begin{equation*}
x(n_1,\ldots,n_m)=y_1^{n_1}\cdots y_k^{n_k}e^{n_{k+1}y_{k+1}+\cdots+n_my_m}
\end{equation*}
for $(y_1,\ldots,y_m)\in[0,\iy)^k\t\R^{m-k}$. This identifies~$X_{\N^k\t\Z^{m-k}}\cong[0,\iy)^k\t\R^{m-k}$.
\smallskip

Using Proposition \ref{bc2prop1} we see that in each of {\bf(i)}--{\bf(iii)}, the topology on $X_P$, and the notions of smooth functions $U\ra\R$, $U\ra(0,\iy)$, $U\ra[0,\iy)$, agree with the usual topology and smooth functions (in the sense of \S\ref{bc211}) on $[0,\iy),\R,[0,\iy)^k\t\R^{m-k}$. Thus, the $X_P$ for general weakly toric monoids $P$ are a class of smooth spaces generalizing the spaces $\R^m_k=[0,\iy)^k\t\R^{m-k}$ used as local models for manifolds with corners.
\end{ex}

\subsubsection{The category $\Mangc$ of manifolds with g-corners}
\label{bc223}

We can now define the category $\Mangc$ of {\it manifolds with generalized corners}, or {\it g-corners}, extending Definition \ref{bc2def2} for the case of ordinary corners.

\begin{dfn}
\label{bc2def11}
Let $X$ be a second countable Hausdorff topological space. An {\it $m$-dimensional generalized chart}, or {\it g-chart}, on $X$ is a triple $(P,U,\phi)$, where $P$ is a weakly toric monoid with $\rank P=m$, and $P$ is a submonoid of $\Z^k$ for some $k\ge 0$, and $U\subseteq X_P$ is open, for $X_P$ as in \S\ref{bc222}, and $\phi:U\ra X$ is a homeomorphism with an open set $\phi(U)$ in~$X$.

Let $(P,U,\phi),(Q,V,\psi)$ be $m$-dimensional g-charts on $X$. We call $(P,U,\phi)$ and $(Q,V,\psi)$ {\it compatible\/} if $\psi^{-1}\ci\phi:\phi^{-1}\bigl(\phi(U)\cap\psi(V)\bigr)\ra
\psi^{-1}\bigl(\phi(U)\cap\psi(V)\bigr)$ is a diffeomorphism between open subsets of $X_P,X_Q$, in the sense of Definition~\ref{bc2def10}.

An $m$-{\it dimensional generalized atlas}, or {\it g-atlas}, on $X$ is a family $\{(P^i,U^i,\phi^i)\!:\!i\!\in\! I\}$ of pairwise compatible $m$-dimensional g-charts on $X$ with $X\!=\!\bigcup_{i\in I}\phi^i(U^i)$. We call such a g-atlas {\it maximal\/} if it is not a proper subset of any other g-atlas. Any g-atlas $\{(P^i,U^i,\phi^i):i\in I\}$ is contained in a unique maximal g-atlas, the family of all g-charts $(P,U,\phi)$ on $X$ compatible with $(P^i,U^i,\phi^i)$ for all~$i\in I$.

An $m$-{\it dimensional manifold with generalized corners}, or {\it g-corners}, is a second countable Hausdorff topological space $X$ with a maximal $m$-dimensional g-atlas. Usually we refer to $X$ as the manifold, leaving the g-atlas implicit. By a {\it g-chart\/ $(P,U,\phi)$ on\/} $X$, we mean an element of the maximal g-atlas. Write~$\dim X=m$.

The {\it interior\/} $X^\ci$ of an $m$-manifold with g-corners $X$ is the dense open subset $X^\ci\subset X$ such that if $(P,U,\phi)$ is a g-chart on $X$ then $\phi^{-1}(X^\ci)=U^\ci$, where $U^\ci\subseteq U\subseteq X_P$ is as in Definition \ref{bc2def10}, so $(P,U^\ci,\phi)$ is a g-chart on~$X^\ci$.

Let $X,Y$ be manifolds with g-corners, and $f:X\ra Y$ a continuous map of the underlying topological spaces. We say that $f:X\ra Y$ is {\it smooth\/} if for all g-charts $(P,U,\phi)$ on $X$ and $(Q,V,\psi)$ on $Y$, the map
\e
\psi^{-1}\ci f\ci\phi: (f\ci\phi)^{-1}(\psi(V))\longra V
\label{bc2eq8}
\e
is a smooth map between the open subsets $(f\ci\phi)^{-1}(\psi(V))\subseteq U\subseteq X_P$ and $V\subseteq X_Q$, in the sense of Definition \ref{bc2def10}. We say that $f:X\ra Y$ is a {\it diffeomorphism\/} if it is a bijection, and both $f:X\ra Y$, $f^{-1}:Y\ra X$ are smooth.

We say that a smooth map $f:X\ra Y$ is {\it interior\/} if $f(X^\ci)\subseteq Y^\ci$. Equivalently, $f$ is interior if the maps \eq{bc2eq8} are interior in the sense of Definition \ref{bc2def10} for all $(P,U,\phi),(Q,V,\psi)$.

Manifolds with g-corners and smooth maps, or interior maps, form a category. Write $\Mangc$ for the category with objects manifolds with g-corners $X,Y$ and morphisms smooth maps $f:X\ra Y$, and $\Mangcin\subset\Mangc$ for the (non-full) subcategory with morphisms interior maps.

By Example \ref{bc2ex1}, we may identify the category $\Manc$ of manifolds with corners with the full subcategory $\Manc\subset\Mangc$ of manifolds with g-corners $X$ which may be covered by charts $(P,U,\phi)$ with $P\cong\N^k\t\Z^{m-k}$ for $0\le k\le m$. Then we have subcategories $\Man\subset\Manb\subset\Manc\subset\Mangc$ and $\Mancin=\Manc\cap\Mangcin$ in $\Mangc$.
\end{dfn}

\begin{ex}
\label{bc2ex2}
If $P$ is a weakly toric monoid, the space $X_P$ in \S\ref{bc222} has a g-atlas $\bigl\{(P,X_P,\id_{X_P})\bigr\}$. This extends to a unique maximal g-atlas, making $X_P$ into a manifold with g-corners, of dimension $\rank P$. All manifolds with g-corners are locally isomorphic to $X_P$ for some $P$. By \cite[Prop.~3.18(a)]{Joyc2}, the interior $X_P^\ci$ is canonically diffeomorphic to $\Hom(P^\gp,\R)\cong\R^{\rank P}$, and is a manifold without boundary. More generally, if $X$ is any manifold with g-corners then $X^\ci$ is a manifold without boundary, that is, it lies in $\Man\subset\Mangc$.
\end{ex}

\begin{ex}
\label{bc2ex3}
Probably the simplest example of a manifold with g-corners which is not an ordinary manifold with corners is the space $X_P$ from the rank 2 toric monoid
\begin{equation*}
P=\bigl\{(a,b,c)\in \Z^2:2a\ge b,\; b\ge 0\bigr\}.
\end{equation*}
Write $p_1=(1,0)$, $p_2=(1,2)$, and $p_3=(1,1)$. Then $p_1,p_2,p_3$ are generators for $P$, subject to the relation $p_1+p_2=2p_3$. Thus Proposition \ref{bc2prop1}(a) gives
\begin{equation*}
X_P\cong \bigl\{(x_1,x_2,x_3)\in[0,\iy)^2:x_1x_2=x_3^2\bigr\}.
\end{equation*}

Topologically, and with depth stratifications, $X_P$ agrees with $[0,\iy)^2$ with coordinates $(x_1,x_2)$. But the smooth structures are different: the function $x_3=\sqrt{x_1x_2}$ is smooth on $X_P$, but not on $[0,\iy)^2$.
\end{ex}

\begin{ex}
\label{bc2ex4}
Another example of a manifold with g-corners which is not a manifold with corners is the space $X_Q$ from the rank 3 toric monoid
\begin{equation*}
Q=\bigl\{(a,b,c)\in \Z^3:a\ge 0,\; b\ge 0,\; a+b\ge c\ge 0\bigr\}.
\end{equation*}
Write $q_1=(1,0,0)$, $q_2=(0,1,1)$, $q_3=(0,1,0)$, $q_4=(1,0,1)$. Then $q_1,q_2,q_3,q_4$ are generators for $Q$, subject to the relation $q_1+q_2=q_3+q_4$. Hence by Proposition \ref{bc2prop1}(a) we may identify
\e
X_Q\cong \bigl\{(x_1,x_2,x_3,x_4)\in[0,\iy)^4:x_1x_2=x_3x_4\bigr\}.
\label{bc2eq9}
\e

\begin{figure}[htb]
\centerline{$\splinetolerance{.8pt}
\begin{xy}
0;<1mm,0mm>:
,(0,0)*{\bu}
,(14,0)*{(0,0,0,0) }
,(-30,-20)*{\bu}
,(-39,-18)*{(x_1,0,0,0)}
,(30,-20)*{\bu}
,(39.5,-18)*{(0,0,x_3,0)}
,(-30,-30)*{\bu}
,(-39,-32)*{(0,0,0,x_4)}
,(30,-30)*{\bu}
,(39.5,-32)*{(0,x_2,0,0)}
,(0,-30)*{\bu}
,(0,-33)*{(0,x_2,0,x_4)}
,(0,-20)*{\bu}
,(0,-17)*{(x_1,0,x_3,0)}
,(-30,-25)*{\bu}
,(-40.5,-25)*{(x_1,0,0,x_4)}
,(30,-25)*{\bu}
,(40.5,-25)*{(0,x_2,x_3,0)}
,(0,0);(-45,-30)**\crv{}
?(.8444)="aaa"
?(.85)="bbb"
?(.75)="ccc"
?(.65)="ddd"
?(.55)="eee"
?(.45)="fff"
?(.35)="ggg"
?(.25)="hhh"
?(.15)="iii"
?(.05)="jjj"
,(0,0);(45,-30)**\crv{}
?(.8444)="aaaa"
?(.85)="bbbb"
?(.75)="cccc"
?(.65)="dddd"
?(.55)="eeee"
?(.45)="ffff"
?(.35)="gggg"
?(.25)="hhhh"
?(.15)="iiii"
?(.05)="jjjj"
,(0,0);(-40,-40)**\crv{}
?(.95)="a"
?(.85)="b"
?(.75)="c"
?(.65)="d"
?(.55)="e"
?(.45)="f"
?(.35)="g"
?(.25)="h"
?(.15)="i"
?(.05)="j"
,(0,0);(40,-40)**\crv{}
?(.95)="aa"
?(.85)="bb"
?(.75)="cc"
?(.65)="dd"
?(.55)="ee"
?(.45)="ff"
?(.35)="gg"
?(.25)="hh"
?(.15)="ii"
?(.05)="jj"
,"a";"aa"**@{.}
,"b";"bb"**@{.}
,"c";"cc"**@{.}
,"d";"dd"**@{.}
,"e";"ee"**@{.}
,"f";"ff"**@{.}
,"g";"gg"**@{.}
,"h";"hh"**@{.}
,"i";"ii"**@{.}
,"j";"jj"**@{.}
,"a";"aaa"**@{.}
,"b";"bbb"**@{.}
,"c";"ccc"**@{.}
,"d";"ddd"**@{.}
,"e";"eee"**@{.}
,"f";"fff"**@{.}
,"g";"ggg"**@{.}
,"h";"hhh"**@{.}
,"i";"iii"**@{.}
,"j";"jjj"**@{.}
,"aa";"aaaa"**@{.}
,"bb";"bbbb"**@{.}
,"cc";"cccc"**@{.}
,"dd";"dddd"**@{.}
,"ee";"eeee"**@{.}
,"ff";"ffff"**@{.}
,"gg";"gggg"**@{.}
,"hh";"hhhh"**@{.}
,"ii";"iiii"**@{.}
,"jj";"jjjj"**@{.}
,(-30,-20);(30,-20)**@{--}
,(-30,-20);(-30,-30)**\crv{}
,(-30,-30);(30,-30)**\crv{}
,(30,-30);(30,-20)**\crv{}
\end{xy}$}
\caption{3-manifold with g-corners $X_Q$ in \eq{bc2eq9}}
\label{bc2fig1}
\end{figure}
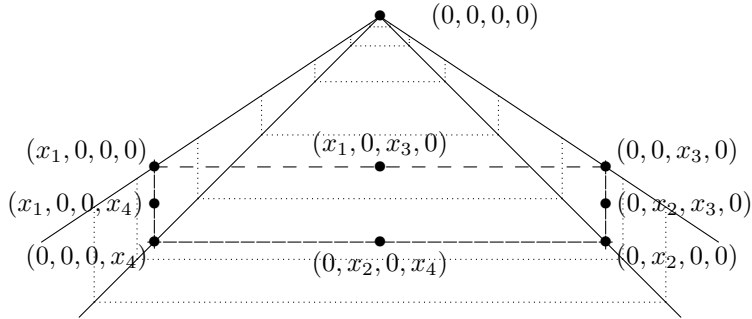

We sketch $X_Q$ in Figure \ref{bc2fig1}. We picture it as a 3-dimensional infinite pyramid on a square base. It has one vertex $(0,0,0,0)$, four 1-dimensional edges of points $(x_1,0,0,0),(0,x_2,0,0), (0,0,x_3,0),\ab(0,\ab 0,\ab 0,\ab x_4)$, four 2-dimensional faces of points $(x_1,0,x_3,0)$, $(x_1,0,0,x_4)$, $(0,x_2,x_3,0)$, $(0,x_2,0,x_4)$, and an interior $X_Q^\ci\ab\cong\R^3$ of points $(x_1,x_2,x_3,x_4)$ with all $x_j\ne 0$. Then $X_Q$ is not a manifold with corners near $(0,0,0,0)$, as we can see from the non-simplicial face structure.
\end{ex}

The next proposition follows from material in \cite[\S 3.6]{Joyc2}.

\begin{prop}
\label{bc2prop2}
Let\/ $X$ be an $m$-manifold with g-corners, and\/ $x\in X$. Then there exists a g-chart\/ $(R,U,\phi)$ on $X$ such that\/ $R\cong P_x\t\Z^{m-k},$ for\/ $P_x$ a \begin{bfseries}toric\end{bfseries} monoid of rank\/ $k$ for $0\le k\le m,$ so that\/ $X_R\cong X_{P_x}\t X_{\Z^{m-k}}\cong X_{P_x}\t\R^{m-k},$ and we have $x=\phi(\de_0,(x_{k+1},\ldots,x_m))$ for some $(\de_0,(x_{k+1},\ldots,x_m))$ in $U\subseteq X_{P_x}\t\R^{m-k},$ where $\de_0\in X_{P_x}$ is the \begin{bfseries}vertex\end{bfseries} of\/ $X_{P_x}$ as in Definition\/ {\rm\ref{bc2def10}}. This $P_x$ depends only on $X,x$ up to isomorphism.
\end{prop}

\subsubsection{B-tangent bundles and b-cotangent bundles}
\label{bc224}
 
In \S\ref{bc212} we explained that manifolds with corners $X$ have two notions of tangent bundle $TX,{}^bTX$. The b-tangent bundle ${}^bTX$ extends nicely to manifolds with g-corners, but as in \cite[Ex.~3.38]{Joyc2} the tangent bundle $TX$ does not. Here is the analogue of Definition~\ref{bc2def3}.

\begin{dfn}
\label{bc2def12}
Let $X$ be an $m$-manifold with g-corners. We can define {\it vector bundles\/} on $X$ in the usual way \cite[Def.~3.37]{Joyc2}. Then $X$ has two natural rank $m$ vector bundles, the {\it b-tangent bundle\/} ${}^bTX\ra X$, and the dual {\it b-cotangent bundle\/} ${}^bT^*X\ra X$. They are defined in~\cite[Def.s 3.39 \& 3.43]{Joyc2}.

If $P$ is a weakly toric monoid, the b-tangent bundle ${}^bTX_P$ of the space $X_P$ in \S\ref{bc222}, regarded as a manifold with g-corners, is canonically trivial with fibre $\Hom(P^\gp,\R)$. More generally, in a g-chart $(P,U,\phi)$ on $X$, ${}^bTX\vert_{\phi(U)}$ is naturally isomorphic to the trivial bundle on $\phi(U)$ with fibre $\Hom(P^\gp,\R)$.

Write $\Ga^\iy({}^bTX)$ for the vector space of smooth sections of ${}^bTX$. Elements of $\Ga^\iy({}^bTX)$ are called {\it b-vector fields} on $X$. Write $C^\iy(X)$ for the $\R$-algebra of smooth functions $f:X\ra\R$. Then b-vector fields $v$ on $X$ act as derivations $v:f\mapsto v(f)$ on~$C^\iy(X)$.

On the model space $X_P$, we have smooth functions $\la_p\in C^\iy(X_P)$ for $p\in P$ as in Definition \ref{bc2def10}. For $\al\in\Hom(P^\gp,\R)$, write $v_\al\in\Ga^\iy({}^bTX)$ for the corresponding constant section under the trivialization ${}^bTX_P\cong \Hom(P^\gp,\R)$. Then $v_\al(\la_p)=\al(p)\cdot\la_p$. More generally, given a g-chart $(P,U,\phi)$ on $X$, over the open set $\phi(U)\subset X$ we have smooth functions $\la_p$ and b-vector fields $v_\al$, and the action of b-vector fields on functions again satisfies~$v_\al(\la_p)=\al(p)\cdot\la_p$.

There is a {\it Lie bracket\/} $[\,,\,]$ on $\Ga^\iy({}^bTX)$, such that $[v,w]$ acts on $f\in C^\iy(X)$ as $[v,w](f)=v(w(f))-w(v(f))$. On $X_P$ we have $[v_\al,v_\be]=0$ for $\al,\be\in\Hom(P^\gp,\R)$, and more generally for $\al,\be\in \Hom(P^\gp,\R)$ and $p,q\in P$\begin{equation*}
\bigl[\la_pv_\al,\la_qv_\be\bigl]=\la_{p+q}\bigl(\al(q)v_\be-\be(p)v_\al\bigr).
\end{equation*}
The same formulae work over a g-chart $(P,U,\phi)$ on $X$.

There is a {\it de Rham b-differential\/} ${}^b\d:C^\iy(X)\ra\Ga^\iy({}^bT^*X)$ satisfying the usual Leibniz rule, which has $v\cdot{}^b\d f=v(f)$ for $v\in\Ga^\iy({}^bTX)$ and $f\in C^\iy(X)$.

Let $f:X\ra Y$ be an interior map of manifolds with g-corners. There are natural vector bundle morphisms ${}^bTf:{}^bTX\ra f^*({}^bTY)$ and ${}^bT^*f:f^*({}^bT^*Y)\ra {}^bT^*X$, the {\it b-derivatives\/} of $f$.

If $X$ lies in $\Manc\subset\Mangc$ then ${}^bTX,{}^bT^*X$ above are canonically isomorphic to those in Definition \ref{bc2def3}, and all the structures on ${}^bTX,{}^bT^*X$ above agree with those in Definition~\ref{bc2def3}.
\end{dfn}

\subsubsection{The depth stratification, and (b-)normal bundles}
\label{bc225}

The next definition, an analogue of Definition \ref{bc2def4}, is taken from~\cite[Def.~3.26]{Joyc2}.

\begin{dfn}
\label{bc2def13}
Let $P$ be a weakly toric monoid, and $F$ a face of $P$, as in Definition \ref{bc2def9}. For $X_F,X_P$ as in Definition \ref{bc2def10}, define an inclusion map $i_F^P:X_F\hookra X_P$ by $i_F^P(y)=\bar y$, where $y\in X_F$ so that $y:F\ra[0,\iy)$ is a monoid morphism, and $\bar y:P\ra[0,\iy)$ is defined by
\begin{equation*}
\bar y(p)=\begin{cases} y(p), & p\in F, \\ 0, & p\in P\sm F. \end{cases}
\end{equation*}
The condition in Definition \ref{bc2def9} that if $p,q\in P$ with $p+q\in F$ then $p,q\in F$ implies that $\bar y$ is a monoid morphism, so $\bar y\in X_P$. Then $i_F^P:X_F\ra X_P$ is a smooth, injective map of manifolds with g-corners.

For each $x\in X_P$, define the {\it support\/} of $x$ to be
\begin{equation*}
\supp x=\bigl\{p\in P: x(p)\ne 0\bigr\}.
\end{equation*}
It is easy to see that $\supp x$ is a face of $P$. For each face $F$ of $P$, write
\begin{equation*}
X^P_F=\bigl\{x\in X_P:\supp x=F\bigr\}.
\end{equation*}
Then the interior $X_P^\ci$ is $X^P_P$, and we have a decomposition
\e
X_P=\coprod\nolimits_{\text{faces $F$ of $P$}}X^P_F.
\label{bc2eq10}
\e

From the definition of $i_F^P:X_F\hookra X_P$, it is easy to see that
\e
X^P_F=i_F^P(X_F^\ci)\quad\text{and}\quad \ov{X^P_F}=i_F^P(X_F)=\coprod\nolimits_{\text{faces $G$ of $P$ with $G\subseteq F$}}X^P_G,
\label{bc2eq11}
\e
where $\ov{X^P_F}$ is the closure of $X^P_F$ in $X_P$. As in Example \ref{bc2ex2} we have a diffeomorphism $X^P_F\cong X_F^\ci \cong \R^{\rank F}=\R^{\rank P-\codim F}$. Thus \eq{bc2eq10} is a locally closed stratification of $X_P$ into smooth manifolds without boundary.

For $x\in X_P$, define the {\it depth\/} $\depth_{X_P}x$ to be $\codim(\supp x)=\rank P-\rank(\supp x)$, so that $\depth_{X_P}x=0,\ldots,\dim X_P$. For each $k=0,\ldots,\dim X_P$, define the {\it depth\/ $k$ stratum\/} of $X_P$ to be
\begin{equation*}
S^k(X_P)=\bigl\{x\in X_P:\depth_{X_P}x=k\bigr\}.
\end{equation*}
Then the interior $X_P^\ci$ is $S^0(X_P)$, and
\begin{equation*}
S^k(X_P)=\coprod\nolimits_{\text{faces $F$ of $P$: $\codim F=k$}}X^P_F,
\end{equation*}
so that $S^k(X_P)$ is a smooth manifold without boundary of dimension $\dim X_P-k$, and \eq{bc2eq11} implies that $\overline{S^k(X_P)}=\bigcup_{l=k}^{\dim X_P}S^l(X_P)$. Hence
\begin{equation*}
X_P=\coprod\nolimits_{k=0}^{\dim X_P}S^k(X_P)
\end{equation*}
is a locally closed stratification of $X_P$ into smooth manifolds without boundary.

Let $X$ be an $m$-manifold with g-corners. For $x\in X$, choose a g-chart $(P,U,\phi)$ on $X$ with $\phi(u)=x$ for $u\in U$, and define the {\it depth\/} $\depth_Xx$ of $x$ in $X$ by $\depth_Xx=\depth_{X_P}u$. This is independent of the choice of $(P,U,\phi)$. For each $k=0,\ldots,m$, define the {\it depth\/ $k$ stratum\/} of $X$ to be
\begin{equation*}
S^k(X)=\bigl\{x\in X:\depth_Xx=k\bigr\}.
\end{equation*}
Then $X=\coprod_{k=0}^mS^k(X)$. Each $S^k(X)$ is a manifold without boundary of dimension $m-k$, with $S^0(X)=X^\ci$, and $\overline{S^k(X)}=\bigcup_{l=k}^mS^l(X)$, since this holds for the local models~$X_P$.
\end{dfn}

Here is the analogue of Definition~\ref{bc2def5}.

\begin{dfn} 
\label{bc2def14}
Let $X$ be an $m$-manifold with g-corners, and $k=0,\ldots,m$. Then as for \eq{bc2eq3}, it is proved in \cite[\S 3.6]{Joyc2} that we have a natural exact sequence of vector bundles on~$S^k(X)$
\e
\xymatrix@C=20pt{ 0 \ar[r] & {}^bN_{S^k(X)} \ar[rr]^{{}^bi} && {}^bTX\vert_{S^k(X)}
\ar[rr]^{{}^b\pi} && T(S^k(X)) \ar[r] & 0, }
\label{bc2eq12}
\e
where ${}^bN_{S^k(X)}$ is the {\it b-normal bundle\/} of $S^k(X)$ in $X$, of rank $k$. Furthermore, ${}^bN_{S^k(X)}\ra S^k(X)$ has a {\it natural flat connection\/} $\nabla^{{}^bN_{S^k(X)}}$, and there is a {\it monoid bundle\/} $M_{S^k(X)}\subset{}^bN_{S^k(X)}$ such that $M_{S^k(X)}\ra X$ is a local system of rank $k$ toric monoids over $X$, such that ${}^bN_{S^k(X)}\cong M_{S^k(X)}\ot_\N\R$, and the flat connection on ${}^bN_{S^k(X)}$ is induced from the local system structure on~$M_{S^k(X)}$.

We can characterize these locally near $x\!\in\! S^k(X)$ as follows: Proposition \ref{bc2prop2} gives a g-chart $(P,U,\phi)$ on $X$ with $P\cong P_x\t\Z^{m-k}$ for $P_x$ a rank $k$ toric monoid, so that $X_P\cong X_{P_x}\t X_{\Z^{m-k}}\cong X_{P_x}\t\R^{m-k},$ and $x=\phi(\de_0,(x_{k+1},\ldots,x_m))$ for some $(\de_0,(x_{k+1},\ldots,x_m))$ in $U\subseteq X_{P_x}\t\R^{m-k},$ where $\de_0\in X_{P_x}$ is the vertex of $X_{P_x}$ as in Definition \ref{bc2def10}. Then $S^k(X)\cap\Phi(U)=\Phi\bigl(U\cap(\{\de_0\}\t\R^{m-k})\bigr)$, giving a local identification $S^k(X)\cong\R^{m-k}$, consistent with $S^k(X)$ being a manifold without boundary of dimension $m-k$. Also on $S^k(X)\cap\Phi(U)$, equation \eq{bc2eq12} is canonically identified with the exact sequence of trivial vector bundles
\e
\xymatrix@C=12pt{ 0 \ar[r] & \Hom(P_x^\gp,\R) \ar[rr]^(0.38){{}^bi} && {\begin{subarray}{l} \ts \quad \Hom(P^\gp,\R)\cong \\ \ts\Hom(P_x^\gp,\R)\op\R^{m-k}\end{subarray}}
\ar[rr]^(0.67){{}^b\pi} && \R^{m-k} \ar[r] & 0. }
\label{bc2eq13}
\e
The flat connection on ${}^bN_{S^k(X)}$ is compatible on $S^k(X)\cap\Phi(U)$ with the trivialization ${}^bN_{S^k(X)}\cong \Hom(P_x^\gp,\R)$. Under the identification between \eq{bc2eq12} and \eq{bc2eq13}, the monoid bundle $M_{S^k(X)}$ on $S^k(X)\cap\Phi(U)$ is identified with $P_x^\vee=\Hom(P_x,\N)\subset \Hom(P_x,\R)=\Hom(P_x^\gp,\R)$, where $P_x^\vee$ is the dual monoid from Definition \ref{bc2def8}. Note that $(P_x^\vee)^\vee\cong P_x$ as $P_x$ is toric, so $P_x\cong (M_{S^k(X)}\vert_x)^\vee$.
\end{dfn}

\subsubsection{A natural Lie algebroid on $S^k(X)$}
\label{bc226}

Here is the g-corners analogue of Definition~\ref{bc2def7}.

\begin{dfn}
\label{bc2def16}
Let $X$ be an $m$-manifold with corners, and $k=0,\ldots,m$. Then there is a natural Lie algebroid on the depth $k$ stratum $S^k(X)$, defined as follows. We take the vector bundle $A$ to be ${}^bTX\vert_{S^k(X)}$, and the anchor map $\rho$ to be ${}^b\pi:{}^bTX\vert_{S^k(X)}\ra T(S^k(X))$ as in \eq{bc2eq12}. If $u,v\in \Ga^\iy({}^bTX\vert_{S^k(X)})$, we choose smooth extensions $\ti u,\ti v$ of $u,v$ to an open neighbourhood $V$ of $S^k(X)$ in $X$ with $\ti u\vert_{S^k(X)}=u$ and $\ti v\vert_{S^k(X)}=v$, and we define the Lie bracket $[\,,\,]_{S^k}$ on $\Ga^\iy({}^bTX\vert_{S^k(X)})$ by $[u,v]_{S^k}=[\ti u,\ti v]\vert_{S^k(X)}$, using the Lie bracket on $\Ga^\iy({}^bTV)$ from \S\ref{bc224}. One can show this is well defined. In the local trivialization ${}^bTX\vert_{S^k(X)}\!\cong\!\Hom(P_x^\gp,\R)\!\op\!\R^{m-k}$ of \eq{bc2eq12}--\eq{bc2eq13} near $x\!\in\! S^k(X)$, we have
\begin{align*}
&\Bigl[\la\op u_{k+1}\frac{\pd}{\pd x_{k+1}}+\cdots+u_m\frac{\pd}{\pd x_m},\mu\op v_{k+1}\frac{\pd}{\pd x_{k+1}}+\cdots+v_m\frac{\pd}{\pd x_m}\Bigr]_{S^k}\\
&=\sum_{i=k+1}^m\Bigl(u_i\frac{\pd \mu}{\pd x_i}-v_i\frac{\pd \la}{\pd x_i}\Bigr)
\op\sum_{i,j=k+1}^m\Bigl(u_i\frac{\pd v_j}{\pd x_i}-v_i\frac{\pd u_j}{\pd x_i}\Bigr)\frac{\pd}{\pd x_j},
\end{align*}
where $(x_{k+1},\ldots,x_m)$ are local coordinates  on $S^k(X)$, and $\la,\mu$ local functions $S^k(X)\ra \Hom(P_x^\gp,\R)$, and $u_i,v_j$ local functions of the coordinates $(x_{k+1},\ab\ldots,\ab x_m)$, and we write $\frac{\pd}{\pd x_{k+1}},\ldots,\frac{\pd}{\pd x_m}$ for the basis of $\R^{m-k}$ in \eq{bc2eq13}, to make our notation compatible with~\eq{bc2eq5}.
\end{dfn}

\section{B-complex manifolds with g-corners}
\label{bc3}

\subsection{B-complex structures on manifolds with g-corners}
\label{bc31}

We generalize some important definitions in complex geometry, including Definition \ref{bc1def1}, to manifolds with (g-)corners.

\begin{dfn}
\label{bc3def1}
Let $X$ be a manifold of with corners or g-corners, of real dimension $2n$. A {\it b-almost complex structure\/} $J$ on $X$ is a morphism of vector bundles $J:{}^bTX\ra{}^bTX$ such that 
\begin{itemize}
\setlength{\itemsep}{0pt}
\setlength{\parsep}{0pt}
\item[(i)] $J^2=-\id_{{}^bTX}$, and
\item[(ii)] For all $0\le k\le 2n$, we have the depth stratum $S^k(X)\subset X$ as in \S\ref{bc213} and \S\ref{bc225}, and from the exact sequences \eq{bc2eq3} and \eq{bc2eq12} we have a rank $k$ vector subbundle ${}^bi\bigl({}^bN_{S^k(X)}\bigr)\subseteq {}^bTX\vert_{S^k(X)}$. We require that in subbundles of ${}^bTX\vert_{S^k(X)}$ we have the transversality condition
\begin{equation*}
{}^bi\bigl({}^bN_{S^k(X)}\bigr)\cap J\bigl[{}^bi\bigl({}^bN_{S^k(X)}\bigr)\bigr]=0.
\end{equation*}
This is only possible if $S^k(X)=\es$ when $n<k\le 2n$. It holds automatically if $k=0$ or~1.
\end{itemize}
We call $(X,J)$ a {\it b-almost complex manifold}.

The usual proof of the existence of Nijenhuis tensors \cite[\S IX.2]{KoNo} also works for b-(co)tangent bundles. Thus, a b-almost complex structure has a {\it Nijenhuis tensor\/} $N_J\in\Ga^\iy({}^bT^*X\ot {}^bT^*X\ot {}^bTX)$ characterized uniquely by the fact that for all b-vector fields $v,w\in\Ga^\iy({}^bTX)$, we have 
\e
N_J(v,w)=[v,w]+J\bigl([Jv,w]+[v,Jw]\bigr)-[Jv, Jw],
\label{bc3eq1}
\e
where $[\,,\,]$ is the Lie bracket of b-vector fields from \S\ref{bc212} and \S\ref{bc224}. We call $J$ {\it integrable\/} if $N_J=0$. If $J$ is integrable we call it a {\it b-complex structure}, and we call $(X,J)$ a {\it b-complex manifold}. 

Equivalently, the complexification ${}^bTX\ot_\R\C$ has a decomposition
\e
{}^bTX \ot_\R \C =  {}^bT^{1,0}X\op{}^bT^{0,1}X
\label{bc3eq2}
\e
into the $i$-eigenspace ${}^bT^{1,0}X $ and $(-i)$-eigenspace ${}^bT^{0,1}X $ of $J$. Then $J$ is integrable if and only if $\bigl[\Ga^\iy({}^bT^{1,0}X),\Ga^\iy({}^bT^{1,0}X)\bigr]\subseteq \Ga^\iy({}^bT^{1,0}X)\subseteq\Ga^\iy({}^bTX\ot_\R\C)$. Dual to \eq{bc3eq2}, we have a decomposition
\e
{}^bT^*X\ot_\R\C={}^bT^{*\,1,0}X\op{}^bT^{*\,0,1}X,
\label{bc3eq3}
\e
where ${}^bT^{*\,1,0}X$ is the complex vector bundle dual to ${}^bT^{1,0}X$, and ${}^bT^{*\,0,1}X$ is dual to ${}^bT^{0,1}X$. Let $f:X\ra\C$ be smooth, so that ${}^b\d f\in\Ga^\iy({}^bT^*X\ot_\R\C)$. Write ${}^b\d f={}^b\pd f\op{}^b\db f$ in the splitting \eq{bc3eq3}, where ${}^b\pd f$ is the component in ${}^bT^{*\,1,0}X$ and ${}^b\db f$ the component in ${}^bT^{*\,0,1}X$. We call $f$ {\it holomorphic\/} if~${}^b\db f=0$.

As in standard complex geometry, ${}^b\db$ naturally extends to maps
\begin{equation*}
{}^b\db \colon  \Ga^\iy\bigl(\La^r({}^bT^{*\,0,1}X)\bigr)\longra \Ga^\iy\bigl(\La^{r+1}({}^bT^{*\,0,1}X)\bigr),
\end{equation*}
for $r=0,1,\ldots,n-1$, satisfying the graded Leibniz rule:
\begin{equation*} 
{}^b \db (\al \wedge \beta) = {}^b \db \al \wedge \beta + (-1)^k \al \wedge {}^b \db \beta \,, 
\end{equation*}
for $\al\in\Ga^\iy\bigl(\La^k({}^bT^{*\,0,1}X)\bigr)$ and $\be\in\Ga^\iy\bigl(X,\La^l({}^bT^{*\,0,1}X)\bigr)$ for $k,l\ge 0$ with $k+l<n$. Then $J$ is integrable if and only if ${}^b\db ^2={}^b\db\ci{}^b\db=0$.

Let $(X,J_X),(Y,J_Y)$ be b-complex manifolds with g-corners and $f:X\ra Y$ be an interior map. We call $f$ {\it holomorphic\/} if for all $x\in X$ with $y=f(x)\in Y$ , so that ${}^b\d f\vert_x \colon {}^bT_x X \ra {}^bT_y Y$ is a linear map, we have 
\begin{equation*} 
{}^b\d f\vert_x \ci J_X\vert_x  = J_Y\vert_y \ci  {}^b\d f\vert_x. 
\end{equation*}
\end{dfn}

\begin{rem}
\label{bc3rem1}
{\bf(a)} For manifolds with boundary, which are equivalent to manifolds with (g-)corners $X$ with $S^k(X)=\es$ for $k>1$, b-(almost) complex structures were defined and studied by Mendoza \cite{Mend1,Mend2}. See also Melrose \cite[\S 6.3]{Melr2} who discusses Riemann surfaces with boundary.  
\smallskip

\noindent{\bf(b)} Definition \ref{bc3def1}(ii) ensures good behaviour of b-(almost) complex manifolds at their corner strata. We have decided to include it from the beginning, as it seems fundamental to the theory, and is necessary for Theorems \ref{bc3thm2}--\ref{bc3thm3} below. It does not appear in \cite{Mend1,Mend2,Melr2}, as it is trivial for manifolds with boundary.
\end{rem}

Here is an important example.

\begin{ex}
\label{bc3ex1}
Let $P$ be a weakly toric monoid, of rank $n$. Then Example \ref{bc2ex2} defines an $n$-manifold with g-corners $X_P$, where if $P=\N^k\t\Z^{n-k}$ then $X_P\cong\R^n_k=[0,\iy)^k\t\R^{n-k}$ is a manifold with corners. As in \S\ref{bc224}, the b-tangent bundle ${}^bTX_P$ is canonically trivial with fibre~$\Hom(P^\gp,\R)\cong\R^n$.

We will define a b-complex structure on the product $X_P\t\Hom(P^\gp,\R)$, which has real dimension $2n$. Since ${}^bT\bigl(X_P\t\Hom(P^\gp,\R)\bigr)$ is canonically trivial with fibre $\Hom(P^\gp,\R)\op\Hom(P^\gp,\R)$, we define $J$ in matrix form by
\e
J=\begin{pmatrix} 0 && -\id \\ \\ \id && 0 \end{pmatrix} \, : \,\,\begin{matrix} \Hom(P^\gp,\R) \\ \op \\ \Hom(P^\gp,\R)\end{matrix} \, \longra \, \begin{matrix} \Hom(P^\gp,\R) \\ \op \\ \Hom(P^\gp,\R).\!\end{matrix}
\label{bc3eq4}
\e

As in Definition \ref{bc2def12}, the trivialization ${}^bTX_P\cong\Hom(P^\gp,\R)$ gives b-vector fields $v_\al$ on $X_P$ for $\al\in\Hom(P^\gp,\R)$ with $[v_\al,v_\be]=0$ for all $\al,\be\in \Hom(P^\gp,\R)$. Write $w_\al\in\Ga^\iy(T\Hom(P^\gp,\R))$ for the vector field on the vector space $\Hom(P^\gp,\R)$ corresponding to translation in direction $\al\in\Hom(P^\gp,\R)$. Then $[w_\al,w_\be]=0$ for all $\al,\be\in\Hom(P^\gp,\R)$. Write $v_\al',w_\al'$ for the vector fields on $X_P\t\Hom(P^\gp,\R)$ obtained by lifting $v_\al,w_\al$ from the factors $X_P,\Hom(P^\gp,\R)$. Then $[v'_\al,v'_\be]=[v'_\al,w'_\be]=[w'_\al,w'_\be]=0$ for all $\al,\be\in\Hom(P^\gp,\R)$. Equation \eq{bc3eq4} means that
\begin{equation*}
J(v_\al')=w_\al',\quad J(w_\al')=-v_\al'\quad\text{for $\al\in\Hom(P^\gp,\R)$.}
\end{equation*}

In the expression \eq{bc3eq1} for the Nijenhuis tensor $N_J$, we find that $N_J(v,w)=0$ if $v,w$ are $v'_\al,v'_\be,w'_\al$ or $w'_\be$, as all four terms on the right hand side are zero. Hence $N_J=0$, and $J$ is integrable. Thus $\bigl(X_P\t\Hom(P^\gp,\R),J\bigr)$ is a b-complex manifold. We call $J$ the {\it standard b-complex structure\/} on $X_P\t\Hom(P^\gp,\R),$ and also write it as $J_\rst$.

As in Definition \ref{bc2def10} if $p\in P$ we have a smooth function $\la_p:X_P\ra[0,\iy)\subset\R$. Define $\th_p:\Hom(P^\gp,\R)\cong\Hom(P,\R)\ra\R$ by $\th_p(\al)=\al(p)$. Combine these to give a function
\begin{equation*}
\la_pe^{i\th_p}=(\la_p\ci\Pi_{X_P})\cdot e^{i\th_p\ci\Pi_{\Hom(P,\R)}}:X_P\t\Hom(P^\gp,\R)\longra\C,
\end{equation*}
where $\Pi_{X_P},\Pi_{\Hom(P,\R)}$ are the projections to the factors of $X_P\t\Hom(P^\gp,\R)$.

On $X_P$ we have $v_\al(\la_p)=\al(p)\la_p$ for $\al\in\Hom(P^\gp,\R)$ by Definition \ref{bc2def12}, and on $\Hom(P,\R)$ we have $w_\al(e^{i\th_p})=i\al(p)e^{i\th_p}$. Thus lifting to $X_P\t\Hom(P^\gp,\R)$ we see that
\begin{equation*}
v'_\al(\la_pe^{i\th_p})=\al(p)\la_pe^{i\th_p},\quad w'_\al(\la_pe^{i\th_p})=i\al(p)\la_pe^{i\th_p}.
\end{equation*}
Since $J(v'_\al)=w'_\al$ and the $v'_\al,w'_\al$ span ${}^bT\bigl(X_P\t\Hom(P^\gp,\R)\bigr)$, it follows that $\la_pe^{i\th_p}:X_P\t\Hom(P^\gp,\R)\ra\C$ is holomorphic, for all $p\in P$. Note that $\la_pe^{i\th_p}\cdot \la_pe^{i\th_p}=\la_{p+q}e^{i\th_{p+q}}$. Thus, we have constructed an algebra of holomorphic functions on $X_P\t\Hom(P^\gp,\R)$ isomorphic to the monoid algebra $\C[P]$.

Observe that $X_P\t\Hom(P^\gp,\R)\cong\Hom(P,\R_{\geq 0}\t\R)$. We can regard the b-complex structure on $X_P\t\Hom(P^\gp,\R)$ as arising from the b-complex structure on $\R_{\geq 0}\t\R$ for which $(r,\th)\mapsto re^{i\th}$ is a holomorphic function, since $\la_pe^{i\th_p}$ is the composition $\Hom(P,\R_{\geq 0}\t\R)\,{\buildrel-\ci p\over\longra}\,\R_{\geq 0}\t\R\,{\buildrel re^{i\th}\over\longra}\,\C$. Writing $\cS^1=\R/2\pi\Z$ and 
\begin{equation*}
\Hom(P^\gp,\cS^1)=\Hom(P^\gp,\R/2\pi\Z)=\Hom(P^\gp,\R)/2\pi \Hom(P^\gp,\Z),
\end{equation*}
all the above also works for the product $X_P\t\Hom(P^\gp,\cS^1)$, noting that for $p\in P$ (also for $p\in P^\gp$) the function $e^{i\th_p}:\Hom(P^\gp,\R)\ra{\rm U}(1)\subset\C$ descends to $\Hom(P^\gp,\cS^1)$, although the function $\th_p:\Hom(P^\gp,\R)\ra\R$ does not. Then $X_P\t\Hom(P^\gp,\cS^1)\cong\Hom(P,\R_{\geq 0}\t\cS^1)$. The b-complex manifolds $X_P\t\Hom(P^\gp,\cS^1)$ for weakly toric monoids $P$ will be very important in the sequel \cite{ArJo2}, as they are local models for Kato--Nakayama spaces of log smooth log $\C$-schemes and log smooth log complex analytic spaces.
\end{ex}

\subsection{\texorpdfstring{Newlander--Nirenberg type theorems for \\ b-complex manifolds with (g-)corners}{Newlander--Nirenberg type theorems for b-complex manifolds with (g-)corners}}\label{bc32}

Mendoza \cite[Prop.~5.1]{Mend2} proves the following:

\begin{thm}
\label{bc3thm1}
Let $(X,J)$ be a b-complex manifold with boundary of real dimension $2n$ for $n\ge 1,$ and suppose  $x\in\pd X=S^1(X).$ Then there exist local coordinates $(\la,\th,x_2,\ldots,x_n,y_2,\ldots,y_n)\in[0,\iy)\t\R^{2n-1}$ defined on an open neighbourhood\/ $U$ of $x$ in $X,$ where $\la$ takes values in $[0,\iy),$ with\/ $\pd X$ locally defined by $\la=0$ and\/ $\la(x)=0,$ such that
\e
{}^b\db(\la e^{i\th}),{}^b\db(x_2+iy_2),\ldots,{}^b\db(x_n+iy_n)\in I_{\la=0}^\iy\cdot\Ga^\iy({}^bT^{*\,0,1}X\vert_U),
\label{bc3eq5}
\e 
where $I_{\la=0}^\iy$ is the ideal of smooth functions $U\ra\C$ vanishing to infinite order along the boundary hypersurface\/ $\la=0$. On $\pd X\cap U,$ the complex\/ $1$-forms in \eq{bc3eq5} are a basis of sections of\/~${}^bT^{*\,0,1}X\vert_{\pd X\cap U}$.
\end{thm}

Here $\la e^{i\th},x_2+iy_2,\ldots,x_n+iy_n$ may not be holomorphic, but they are at least `holomorphic to infinite order at $\pd X$'. We can regard Theorem \ref{bc3thm1} as a {\it formal Newlander--Nirenberg Theorem}, as it basically says that the {\it formal completion\/} of $X$ along $\pd X=S^1(X)$ is locally standard.

\begin{rem}
\label{bc3rem2}
Barron--Francis \cite[Th.~1.2]{BaFr} claim a stronger result, that one can choose coordinates with ${}^b\db(\la e^{i\th})={}^b\db(x_2+iy_2)=\cdots={}^b\db(x_n+iy_n)=0$, that is, $\la e^{i\th},\ldots,x_n+iy_n$ are holomorphic functions on $U$. Unfortunately there is a mistake in the proof on \cite[p.~12]{BaFr} (in arXiv version 1, 2023), in the sentence `\dots there is no harm in assuming the deformation terms to be $\ti\si$-invariant'. Barron and Francis have told the authors they hope to post a corrected version.
\end{rem}

Here is our first main result, proved in \S\ref{bc33}. It shows that on each depth stratum $S^k(X)$ of a b-complex manifold with (g-)corners $X$, the action of $J$ on the Lie algebroid ${}^bTX\vert_{S^k(X)}$ has a standard local model.

\begin{thm}
\label{bc3thm2}
Let\/ $(X,J)$ be a b-complex manifold with corners or g-corners of real dimension $2n,$ and let\/ $x\in S^k(X)$ for\/ $0\le k\le n$. Then we can choose an open neighbourhood\/ $V$ of\/ $x$ in $S^k(X),$ global coordinates $(\th_1,\ldots,\th_k,x_{k+1},\ab\ldots,\ab x_n,\ab y_{k+1},\ab\ldots,y_n)$ on $V,$ and a basis of sections\/ $v_1,\ldots,v_k,\frac{\ti\pd}{\pd\th_1},\ldots,\frac{\ti\pd}{\pd\th_k},\ab\frac{\ti\pd}{\pd x_{k+1}},\ab\ldots,\ab\frac{\ti\pd}{\pd x_n},\ab\frac{\ti\pd}{\pd y_{k+1}},\ab\ldots,\frac{\ti\pd}{\pd y_n}$ of\/ ${}^bTX\vert_V,$ satisfying
\begin{itemize}
\setlength{\itemsep}{0pt}
\setlength{\parsep}{0pt}
\item[{\bf(i)}] Using the notation of\/ \eq{bc2eq3} and\/ {\rm\eq{bc2eq12},} $v_1,\ldots,v_k$ are a basis of sections of\/ ${}^bN_{S^k(X)}\vert_V,$ and are constant under the natural flat connection $\nabla^{{}^bN_{S^k(X)}}$. 
\item[{\bf(ii)}] All elements of the basis of sections $v_1,\ldots,\frac{\ti\pd}{\pd y_n}$ commute with each other under the Lie bracket\/ $[\,,\,]_{S^k}$ in {\rm\S\ref{bc214}} and\/~{\rm\S\ref{bc226}}.
\item[{\bf(iii)}] The morphisms ${}^b\pi$ in \eq{bc2eq3} and\/ \eq{bc2eq12} map $\frac{\ti\pd}{\pd\th_i}\mapsto\frac{\pd}{\pd\th_i},$ $\frac{\ti\pd}{\pd x_j}\mapsto\frac{\pd}{\pd x_j}$ and\/ $\frac{\ti\pd}{\pd y_j}\mapsto\frac{\pd}{\pd y_j}$ for all\/~$i,j$.
\item[{\bf(iv)}] The restriction of\/ $J$ to ${}^bTX\vert_V$ satisfies for all\/ $1\le i\le k$ and\/~$k<j\le n$
\end{itemize}
\e
J(v_i)=\frac{\ti\pd}{\pd\th_i},\quad  J\Bigl(\frac{\ti\pd}{\pd\th_i}\Bigr)=-v_i,\quad
J\Bigl(\frac{\ti\pd}{\pd x_j}\Bigr)=\frac{\ti\pd}{\pd y_j}, \quad
J\Bigl(\frac{\ti\pd}{\pd y_j}\Bigr)=-\frac{\ti\pd}{\pd x_j}. 
\label{bc3eq6}
\e

Write\/ $z_j=x_j+iy_j:V\ra\C$ for\/ $j=k+1,\ldots,n$. Then under the map
\begin{equation*}
{}^b\pi^*:T^*V\ot_\R\C\longra {}^bT^*X\vert_V\ot_\R\C,
\end{equation*}
we have ${}^b\pi^*({}^b\d z_j)\in\Ga^\iy\bigl({}^bT^{*\,1,0}X\vert_V\bigr)\subset\Ga^\iy\bigl({}^bT^*X\vert_V\ot_\R\C\bigr)$. In effect this means that\/ $z_{k+1},\ldots,z_n$ are holomorphic functions on $V\subset S^k(X)$. We can regard\/ $(\th_1,\ldots,\th_k,z_{k+1},\ldots,z_n)$ as coordinates on $V\subset S^k(X)$ in\/~$\R^k\t\C^{n-k}$.
\end{thm}

Here is our second main result, proved in \S\ref{bc34}. It generalizes Theorem \ref{bc3thm1} above from Mendoza \cite{Mend2} to b-complex manifolds with (g-)corners. 

\begin{thm}
\label{bc3thm3}
Let\/ $(X,J)$ be a b-complex manifold with corners or g-corners of real dimension $2n,$ and let\/ $x\in S^k(X)$ for\/ $0\le k\le n$. Define $Q=(M_{S^k(X)}\vert_x)^\vee$ for $M_{S^k(X)}\subset{}^bN_{S^k(X)}$ as in {\rm\S\ref{bc225},} so that\/ $Q$ is a toric monoid of rank\/ $k,$ and set\/ $P=Q\t\Z^{n-k}$. If\/ $X$ is a manifold with corners then $Q\cong\N^k$ and\/ $X_Q\cong[0,\iy)^k$. Example\/ {\rm\ref{bc3ex1}} defines a standard b-complex manifold
\e
\bigl(X_P\t\Hom(P^\gp,\R),J_\rst\bigr)=\bigl((X_Q\t\R^{n-k})\t(\Hom(Q^\gp,\R)\t\R^{n-k}),J_\rst\bigr),
\label{bc3eq7}
\e
with holomorphic functions $\la_pe^{i\th_p}:X_P\t\Hom(P^\gp,\R)\ra\C$ for all\/ $p\in P$. As in\/ {\rm\S\ref{bc222},} since $Q$ is toric we have the vertex $\de_0\in X_Q,$ and
\e
S^k\bigl(X_P\t\Hom(P^\gp,\R)\bigr)=(\{\de_0\}\t\R^{n-k})\t(\Hom(Q^\gp,\R)\t\R^{n-k}).
\label{bc3eq8}
\e

Then there exists a g-chart\/ $(P,U,\phi)$ on $X$ in the sense of\/ {\rm\S\ref{bc223},} such that for all\/ $p\in P,$ the smooth function $\phi_*(\la_pe^{i\th_p}):\phi(U)\ra\C$ defined on the open subset\/ $\phi(U)\subset X$ has
\e
{}^b\db\bigl(\phi_*(\la_pe^{i\th_p})\bigr)\in I_{S^k(X)}^\iy\cdot\Ga^\iy\bigl({}^bT^{*\,0,1}X\vert_{\phi(U)}\bigr),
\label{bc3eq9}
\e
where $I_{S^k(X)}^\iy$ is the ideal of smooth functions $\phi(U)\ra\C$ vanishing to infinite order along\/ $S^k(X)\cap\phi(U)$. So $\phi_*(\la_pe^{i\th_p})$ may not be holomorphic on $(\phi(U),J),$ but it is at least `holomorphic to infinite order at\/~$S^k(X)\cap\phi(U)$'. 

This means that the b-complex structures $J\vert_{\phi(U)}$ and\/ $\phi_*(J_\rst)$ on $\phi(U)\subset X$ agree to infinite order along $S^k(X)\cap\phi(U)$. 
\end{thm}

In fact Theorem \ref{bc3thm2} follows from Theorem \ref{bc3thm3}, but we prove Theorem \ref{bc3thm2} first as we use it in the proof of Theorem \ref{bc3thm3}.

\subsection{Proof of Theorem \ref{bc3thm2}}
\label{bc33}

Work in the situation of  Theorem \ref{bc3thm2}. Consider the complexification of~\eq{bc2eq12}:
\begin{equation}
\xymatrix@C=12pt{ 0 \ar[r] & {}^bN_{S^k(X)} \!\ot\! \C \ar[rr]^(0.45){{}^bi^{\C}} && {}^bTX\vert_{S^k(X)}\! \ot\! \C \ar[rr]^(0.5){{}^b\pi^{\C}} && T(S^k(X)) \!\ot\! \C \ar[r] & 0. }
\label{bc3eq10}
\end{equation}
As in \eq{bc3eq2}, we have a splitting
\e
{}^bTX\vert_{S^k(X)} \ot \C =  {}^bT^{1,0}X\vert_{S^k(X)} \op   {}^bT^{0,1}X\vert_{S^k(X)}
\label{bc3eq11}
\e
into the $\pm i$-eigenspaces of $J$. Let 
\e
\mathcal{W} = {}^b\pi^{\C}( {}^bT^{1,0}X\vert_{S^k(X)}) \quad\text{and}\quad
\mathcal{\overline{W}} = {}^b\pi^{\C}( {}^bT^{0,1}X\vert_{S^k(X)}).
\label{bc3eq12}
\e

Now Definition \ref{bc3def1}(ii) implies that
\begin{equation*}
\Ker({}^b\pi^{\C})\cap \bigl({}^bT^{1,0}X\vert_{S^k(X)}\bigr)=\Ker({}^b\pi^{\C})\cap \bigl({}^bT^{0,1}X\vert_{S^k(X)}\bigr)=0.
\end{equation*}
Thus ${}^b\pi^{\C}:{}^bT^{1,0}X\vert_{S^k(X)}\ra\mathcal{W}$ and ${}^b\pi^{\C}:{}^bT^{0,1}X\vert_{S^k(X)}\ra\mathcal{\overline{W}}$ are isomorphisms. Hence $\mathcal{W},\mathcal{\overline{W}}$ in \eq{bc3eq12} are complex vector subbundles of $TS^k(X)\ot\C$ of rank $n$, which are complex conjugate as ${}^bT^{1,0}X,{}^bT^{0,1}X$ are. From \eq{bc3eq11} we see that
\e
\mathcal{W} + \mathcal{\overline{W}} = TS^k(X) \ot \C .
\label{bc3eq13}
\e

As $TS^k(X)\ot\C$ has rank $2n-k$, we see that $\mathcal{W} \cap \mathcal{\overline{W}}$ is a complex vector subbundle of $TS^k(X)\ot\C$ of rank $k$. We claim that
\e
\mathcal{W} \cap \overline{\mathcal{W}} = {}^b\pi \ci J\vert_{S^k(X)} \ci {}^bi\bigl({}^bN_{S^k(X)}\ot\C\bigr). 
\label{bc3eq14}
\e
To see this, note that if $x'\in S^k(X)$ and $v\in {}^bi\bigl({}^bN_{S^k(X)}\ot\C\bigr)\vert_{x'}$ then
\begin{equation*}
iv+Jv\in {}^bT_{x'}^{1,0}X\vert_{S^k(X)},\quad -iv+Jv\in {}^bT_{x'}^{0,1}X\vert_{S^k(X)}.
\end{equation*}
Applying ${}^b\pi$ and noting that ${}^b\pi(\pm iv)=0$ as \eq{bc3eq10} is exact we see that ${}^b\pi(Jv)\in\mathcal{W}\vert_{x'}$ and ${}^b\pi(Jv)\in\mathcal{\overline{W}}\vert_{x'}$, so ${}^b\pi(Jv)\in(\mathcal{W} \cap \overline{\mathcal{W}})\vert_{x'}$. Thus, the left hand side of \eq{bc3eq14} contains the right hand side. But both are complex vector bundles of rank $k$, from above for $\mathcal{W} \cap \overline{\mathcal{W}}$ and by Definition \ref{bc3def1}(ii), so they must be equal.

\begin{lem}
\label{bc3lem1}
The complex vector subbundles $\mathcal{W},\mathcal{\overline{W}}$ and $\mathcal{W} \cap \mathcal{\overline{W}}$ are \begin{bfseries}involutive\end{bfseries} subbundles of\/ $TS^k(X) \ot \C,$ that is, 
\begin{gather*}
[\Ga^\iy(\mathcal{W}),\Ga^\iy(\mathcal{W})]\subseteq \Ga^\iy(\mathcal{W}),\quad 
[\Ga^\iy(\overline{\mathcal{W}}), \Ga^\iy(\overline{\mathcal{W}})]\subseteq \Ga^\iy(\overline{\mathcal{W}}), \\ 
\text{and}\quad [\Ga^\iy(\mathcal{W} \cap \overline{\mathcal{W}}), \Ga^\iy(\mathcal{W} \cap \overline{\mathcal{W}})]\subseteq \Ga^\iy(\mathcal{W} \cap \overline{\mathcal{W}}).
\end{gather*}
\end{lem}

\begin{proof}
Since $J$ is integrable, ${}^bT^{1,0}X$ and ${}^bT^{0,1}X$ are involutive subbundles of ${}^bTX\ot\C$. So, ${}^bT^{1,0}X\vert_{S^k(X)}$ and ${}^bT^{0,1}X\vert_{S^k(X)}$ are also involutive subbundles of the Lie algebroid ${}^bTX\vert_{S^k(X)}\ot\C$ from \S\ref{bc214} and \S\ref{bc226}. Since ${}^b\pi^{\C}$ in \eq{bc3eq10} is the anchor map of the Lie algebroid ${}^bTX\vert_{S^k(X)}\ot\C$, it is compatible with the Lie brackets. Therefore $\mathcal{W} = {}^b\pi^{\C}( {}^bT^{1,0}X\vert_{S^k(X)})$ and $\overline{\mathcal{W}} = {}^b\pi^{\C}( {}^bT^{0,1}X\vert_{S^k(X)})$ are also involutive, so $\mathcal{W} \cap \overline{\mathcal{W}}$ is by taking intersections.
\end{proof}

Applying Nirenberg's complex Frobenius Theorem \cite{Nire} (see also H\"ormander \cite{Horm}) to these involutive subbundles shows that there exist local real coordinates $(\th'_1,\ldots,\th'_k,x_{k+1},\ldots,x_n,x_{k+1},\ldots,x_n)$ defined on a neighbourhood $V$ of $x$ in $S^k(X)$, with $x=(0,\ldots,0)$ in coordinates, such that
\ea
&\frac{\pd}{\pd \th'_1},\ldots,\frac{\pd}{\pd \th'_k}\quad\text{are a basis of sections of $\mathcal{W} \cap \overline{\mathcal{W}}$,}
\label{bc3eq15}\\
&\frac{\pd}{\pd \th'_1},\ldots,\frac{\pd}{\pd \th'_k},\frac{\pd}{\pd x_{k+1}}-i\frac{\pd}{\pd y_{k+1}},\ldots,\frac{\pd}{\pd x_n}-i\frac{\pd}{\pd y_n}
\quad\text{are a basis of sections of $\mathcal{W}$,}
\nonumber\\
&\frac{\pd}{\pd \th'_1},\ldots,\frac{\pd}{\pd \th'_k},\frac{\pd}{\pd x_{k+1}}+i\frac{\pd}{\pd y_{k+1}},\ldots,\frac{\pd}{\pd x_n}+i\frac{\pd}{\pd y_n}
\quad\text{are a basis of sections of $\overline{\mathcal{W}}$.}
\nonumber
\ea

Making the neighbourhood $V$ of $x$ in $S^k(X)$ smaller if necessary, suppose that $v_1,\ldots,v_k$ are a basis of sections of ${}^bN_{S^k(X)}\vert_V,$ which are constant under the natural flat connection $\nabla^{{}^bN_{S^k(X)}}$, as in Theorem \ref{bc3thm2}(i). Then in the Lie algebroid Lie bracket $[\,,\,]_{S^k}$ on ${}^bTX\vert_V$ from \S\ref{bc214} and \S\ref{bc226} we have $[u,v_i]_{S^k}=0$ for any $u\in\Ga^\iy({}^bTX\vert_V)$, so in particular $[v_i,v_j]_{S^k}=[J(v_i),v_j]_{S^k}=[v_i,J(v_j)]_{S^k}=0$ for all $i,j=1,\ldots,k$. Since $J$ is integrable and $N_J=0$ we see from \eq{bc3eq1} that
\begin{equation*}
[J(v_i),J(v_j)]_{S^k}=0\quad\text{in $\Ga^\iy({}^bTX\vert_V)$ for $i,j=1,\ldots,k$.}
\end{equation*}
Applying the anchor map ${}^b\pi$, which commutes with Lie brackets, shows that
\e
\bigl[{}^b\pi(J(v_i)),{}^b\pi(J(v_j))\bigr]=0\quad\text{in $\Ga^\iy(TV)$ for $i,j=1,\ldots,k$.}
\label{bc3eq16}
\e

By \eq{bc3eq14}--\eq{bc3eq15}, both $\frac{\pd}{\pd\th'_1},\ldots,\frac{\pd}{\pd \th'_k}$ and ${}^b\pi(J(v_1)),\ldots,{}^b\pi(J(v_k))$ are bases of sections of $\mathcal{W} \cap \overline{\mathcal{W}}$ over $\C$. Since $\frac{\pd}{\pd \th'_i}(x_j)=\frac{\pd}{\pd \th'_i}(y_j)=0$ we see that ${}^b\pi(J(v_i))(x_j)={}^b\pi(J(v_i))(y_j)=0$, that is, $x_j,y_j$ are constant along the local flow of the vector fields $v_i$. We claim that making $V$ smaller if necessary, we can change to new coordinates $(\th_1,\ldots,\th_k,x_{k+1},\ldots,x_n,x_{k+1},\ldots,x_n)$ on $V$ with $x=(0,\ldots,0)$ in coordinates, such that \eq{bc3eq15} holds with $\frac{\pd}{\pd\th_i}$ in place of $\frac{\pd}{\pd\th'_i}$, and also
\e
\frac{\pd}{\pd\th_i}={}^b\pi(J(v_i)).
\label{bc3eq17}
\e

To do this, we take $\th_1,\ldots,\th_k:V\ra\R$ to satisfy $\th_1(x')=\cdots=\th_k(x')=0$ whenever $\th'_1(x')=\cdots=\th'_k(x')=0$, and to satisfy \eq{bc3eq17}. We treat this as an existence problem for solutions of a first-order p.d.e.\ \eq{bc3eq17} for $\th_i$, with initial condition that $\th_i=0$ on the codimension $k$ submanifold where $\th'_1=\cdots=\th'_k=0$. Equation \eq{bc3eq16} implies that these p.d.e.s have unique solutions near~$x$.

We can now complete the proof of Theorem \ref{bc3thm2}. Define $\frac{\ti\pd}{\pd\th_i}=J(v_i)$ for $i=1,\ldots,k$. Then ${}^b\pi$ maps $\frac{\ti\pd}{\pd\th_i}\mapsto\frac{\pd}{\pd\th_i},$ as in Theorem \ref{bc3thm2}(iii), and $J(v_i)=\frac{\ti\pd}{\pd\th_i}$ in Theorem \ref{bc3thm2}(iv) holds by definition. As above ${}^b\pi:{}^bT^{1,0}X\vert_{S^k(X)}\ra\mathcal{W}$ is an isomorphism, and $\frac{\pd}{\pd x_j}-i\frac{\pd}{\pd y_j}$ are sections of $\mathcal{W}\vert_V$ for $k<j\le n$. Define $\frac{\ti\pd}{\pd x_j},\frac{\ti\pd}{\pd y_j}$ to be the unique lifts of $\frac{\pd}{\pd x_j},\frac{\pd}{\pd y_j}$ along ${}^b\pi:{}^bTX\vert_V\ra TV$ such that $\frac{\ti\pd}{\pd x_j}-i\frac{\ti\pd}{\pd y_j}$ is the unique section of ${}^bT^{1,0}X\vert_V$ with ${}^b\pi(\frac{\ti\pd}{\pd x_j}-i\frac{\ti\pd}{\pd y_j})=\frac{\pd}{\pd x_j}-i\frac{\pd}{\pd y_j}$. Then ${}^b\pi$ maps $\frac{\ti\pd}{\pd x_j}\mapsto\frac{\pd}{\pd x_j}$ and $\frac{\ti\pd}{\pd y_j}\mapsto\frac{\pd}{\pd y_j}$ for all $j$, as in Theorem \ref{bc3thm2}(iii), and $J$ maps $\frac{\ti\pd}{\pd x_j}\mapsto\frac{\ti\pd}{\pd y_j}$, $\frac{\ti\pd}{\pd y_j}\mapsto -\frac{\ti\pd}{\pd x_j}$ as in Theorem~\ref{bc3thm2}(iv).

We have now constructed all the data in Theorem \ref{bc3thm2}, and proved (i),(iii) and (iv). It remains to prove (ii). As above $[u,v_i]_{S^k}=0$ for any $u\in\Ga^\iy({}^bTX\vert_V)$, so the $v_i$ commute with the whole basis of sections. Since the $\frac{\pd}{\pd\th_i},\frac{\pd}{\pd x_j},\frac{\pd}{\pd y_j}$ commute with each other, and ${}^b\pi$ is a Lie algebra morphism, the commutators of $\frac{\ti\pd}{\pd\th_i},\frac{\ti\pd}{\pd x_j},\frac{\ti\pd}{\pd y_j}$ lie in the kernel of ${}^b\pi$, which is $\an{v_1,\ldots,v_k}_\R$. Noting that $N_J=0$, from \eq{bc3eq1} with $v=v_i$, $w=\frac{\ti\pd}{\pd x_j}$, equation \eq{bc3eq6}, and $[u,v_i]_{S^k}=0$, we see that
\begin{equation*}
\ts J\bigl(\bigl[\frac{\ti\pd}{\pd\th_i},\frac{\ti\pd}{\pd x_j}\bigr]_{S^k}\bigr)-\bigl[\frac{\ti\pd}{\pd\th_i},\frac{\ti\pd}{\pd y_j}\bigr]_{S^k}=0.
\end{equation*}
Thus $[\frac{\ti\pd}{\pd\th_i},\frac{\ti\pd}{\pd x_j}]_{S^k}\!=\![\frac{\ti\pd}{\pd\th_i},\frac{\ti\pd}{\pd y_j}]_{S^k}\!=\!0$, as both lie in $\an{v_1,\ldots,v_k}_\R$ with $\an{v_1,\ldots,v_k}_\R\ab\cap J(\an{v_1,\ldots,v_k}_\R)=0$.

For $k<a,b\le n$, we have
\begin{equation*}
\ts\bigl[\frac{\ti\pd}{\pd x_a}-i\frac{\ti\pd}{\pd y_a},\frac{\ti\pd}{\pd x_b}-i\frac{\ti\pd}{\pd y_b}\bigr]_{S^k}=\bigl[\frac{\ti\pd}{\pd x_a}+i\frac{\ti\pd}{\pd y_a},\frac{\ti\pd}{\pd x_b}+i\frac{\ti\pd}{\pd y_b}\bigr]_{S^k}=0,
\end{equation*}
as these are commutators in ${}^bT^{1,0}X\vert_V$ and ${}^bT^{0,1}X\vert_V$ respectively, and so lie in ${}^bT^{1,0}X\vert_V$ and ${}^bT^{0,1}X\vert_V$, which have zero intersection with $\an{v_1,\ldots,v_k}_\C$.

Define complex functions $\om_{ab}^c:V\ra\C$ by
\e
\ts\bigl[\frac{\ti\pd}{\pd x_a}-i\frac{\ti\pd}{\pd y_a},\frac{\ti\pd}{\pd x_b}+i\frac{\ti\pd}{\pd y_b}\bigr]_{S^k}=\sum_{c=1}^k\om_{ab}^cv_c.
\label{bc3eq18}
\e
We need $\om_{ab}^c=0$ for Theorem \ref{bc3thm2}(ii) to hold, but this may not be true, so we need to add corrections to our coordinates $\th_1,\ldots,\th_k$ to make~$\om_{ab}^c=0$.

Since everything commutes with $\frac{\ti\pd}{\pd\th_d}$, we see that $\frac{\pd}{\pd\th_d}(\om_{ab}^c)=0$ for $d=1,\ldots,k$, so making $V$ smaller if necessary we can suppose $\om_{ab}^c$ is a function of $x_{k+1},y_{k+1},\ab\ldots,\ab x_n,\ab y_n$ only. Set $z_a=x_a+iy_a$, and write $\om_{ab}^c=\om_{ab}^c(z_{k+1},\ldots,z_n)$. For $c=1,\ldots,k$, consider the complex 2-form
\begin{equation*}
\om^c=\sum_{a,b=k+1}^n\om_{ab}^c\d z_a\w\d\bar z_b,
\end{equation*}
considered as defined locally near 0 in $\C^{n-k}$ with coordinates $(z_{k+1},\ab\ldots,\ab z_n)$. From \eq{bc3eq18} we see that $\overline{\om_{ab}^c}=-\om_{ba}^c$, so $\ov{\om^c}=\om^c$, and $\om^c$ is a real (1,1)-form. 

By taking commutators of \eq{bc3eq18} with $\frac{\ti\pd}{\pd x_a},\frac{\ti\pd}{\pd y_a}$ and using the Jacobi identity, we can show $\om^c$ is a closed real (1,1)-form. Thus by the $\pd\db$-Lemma, making $V$ smaller if necessary, we can suppose there exist smooth real functions $f_1,\ldots,f_k$ defined locally near 0 in $\C^{n-k}$, with $f_c(0,\ldots,0)=0$, such that
\e
\ts\om_{ab}^c=i\bigl(\frac{\ti\pd}{\pd x_a}-i\frac{\ti\pd}{\pd y_a}\bigr)\bigl(\frac{\ti\pd}{\pd x_b}+i\frac{\ti\pd}{\pd y_b}\bigr)(f_c),\;\> a,b=k\!+\!1,\ldots,n,\;\> c=1,\ldots,k.
\label{bc3eq19}
\e

Now replace the coordinates $\th_c,x_j,y_j$ on $V$ by
\begin{equation*}
\hat\th_c=\th_c-\ha f_c(x_{k+1}+iy_{k+1},\ldots,x_n+iy_n),\quad \hat x_j=x_j,\quad \hat y_j=y_j.
\end{equation*}
The vector fields $\frac{\pd}{\pd\th_c},\frac{\pd}{\pd x_j},\frac{\pd}{\pd y_j}$ transform by
\begin{equation*}
\ts\frac{\pd}{\pd\hat\th_c}=\frac{\pd}{\pd\th_c},\quad
\frac{\pd}{\pd\hat x_j}=\frac{\pd}{\pd x_j}+\ha\sum_{c=1}^k \frac{\pd f_c}{\pd x_j}\frac{\pd}{\pd\th_c},\quad
\frac{\pd}{\pd\hat y_j}=\frac{\pd}{\pd y_j}+\ha\sum_{c=1}^k \frac{\pd f_c}{\pd y_j}\frac{\pd}{\pd\th_c}.
\end{equation*}
The lifts $\frac{\ti\pd}{\pd\th_c},\frac{\ti\pd}{\pd x_j},\frac{\ti\pd}{\pd y_j}$ to ${}^bTX\vert_V$ transform by
\begin{gather*}
\ts\frac{\ti\pd}{\pd\hat\th_c}=\frac{\ti\pd}{\pd\th_c},\quad
\frac{\ti\pd}{\pd\hat x_j}=\frac{\ti\pd}{\pd x_j}+\ha\sum_{c=1}^k \frac{\pd f_c}{\pd x_j}\frac{\ti\pd}{\pd\th_c}+\ha\sum_{c=1}^k \frac{\pd f_c}{\pd y_j}v_c,\\
\ts\frac{\ti\pd}{\pd\hat y_j}=\frac{\ti\pd}{\pd y_j}+\ha\sum_{c=1}^k \frac{\pd f_c}{\pd y_j}\frac{\ti\pd}{\pd\th_c}-\ha\sum_{c=1}^k \frac{\pd f_c}{\pd x_j}v_c,
\end{gather*}
since this preserves the property \eq{bc3eq6}. Substituting these into \eq{bc3eq18} and using \eq{bc3eq19}, we see that $\om_{ab}^c$ is replaced by $\hat\om_{ab}^c=0$. Thus, the modified coordinates $(\hat\th_1,\ldots,\hat\th_k,\hat x_{k+1},\ab\ldots,\ab\hat x_n,\ab\hat y_{k+1},\ab\ldots,\hat y_n)$ and lifts $\frac{\ti\pd}{\pd\hat\th_c},\frac{\ti\pd}{\pd\hat x_j},\frac{\ti\pd}{\pd\hat y_j}$ also satisfy Theorem \ref{bc3thm2}(ii). The last part, showing $z_{k+1},\ldots,z_n$ are holomorphic, follows from \eq{bc3eq6}. This completes the proof.

\subsection{Proof of Theorem \ref{bc3thm3}}
\label{bc34}

\subsubsection{Set up for the proof}
\label{bc341}

We will use the notation in the following definition throughout the proof.

\begin{dfn}
\label{bc3def2}
As in Theorem \ref{bc3thm3}, we take $(X,J)$ to be a b-complex manifold with g-corners of real dimension $2n$, and fix $x\in S^k(X)$ for $0\le k\le n$. (We take $X$ to have g-corners rather than ordinary corners, as this is more general.) Define $Q=(M_{S^k(X)}\vert_x)^\vee$, so that $Q$ is a toric monoid of rank $k$, set $P=Q\t\Z^{n-k}$, and use the notation in \eq{bc3eq7} and~\eq{bc3eq8}.

During the proof, $W$ will denote an open neighbourhood of $x$ in $X$, and $V=W\cap S^k(X)$, so that $V$ is an open neighbourhood of $x$ in $S^k(X)$. Often we need to make $W,V$ smaller, that is, to replace $W,V$ by open neighbourhoods $W'$ of $x$ in $W$ and $V'=W'\cap S^k(X)$. Then we continue to write $W,V$ for these smaller neighbourhoods. For the g-chart $(P,U,\phi)$ in Theorem \ref{bc3thm3}, we will have $W=\phi(U)$, so that $\phi^{-1}:W\hookra X_P\t\Hom(P^\gp,\R)$.

We will use superscripts $(\cdots)^\sk$ to denote objects defined on $V\subset S^k(X)$. The corresponding symbols $(\cdots)$ without the superscript ${}^\sk$ means an object defined on $W$, whose restriction to $V$ is $(\cdots)^\sk$. So for example, we will have coordinates $(\th_1^\sk,\ldots,\th_k^\sk,z_{k+1}^\sk,\ldots,z_n^\sk)$ on $V$ which are restrictions to $V$ of functions $\th_1,\ldots,\th_k,z_{k+1},\ldots,z_n$ on~$W$.

We will define the data $(P,U,\phi),\phi_*(\la_pe^{i\th_p}),\ldots$ in the theorem by first constructing approximations that do not satisfy all the desired properties, and then making successive corrections. A superscript $(\cdots)^\star$ denotes an object which is an approximation, but still needs to be corrected, so for example we will consider functions $\la_p^\star e^{i\th^\star_p}$ which are approximations to the $\la_pe^{i\th_p}$ in the theorem. 

By Proposition \ref{bc2prop2} for $X,x$, there exists a g-chart $(P,U^\star,\phi^\star)$ on $X$ such that $P=Q\t\Z^{2n-k},$ so that $X_P\cong X_Q\t X_{\Z^{2n-k}}\cong X_Q\t\R^{2n-k},$ and we have $x=\phi^\star(\de_0,(x_{k+1},\ldots,x_{2n}))$ for some $(\de_0,(x_{k+1},\ldots,x_m))$ in $U^\star\subseteq X_Q\t\R^{2n-k},$ where $\de_0\in X_Q$ is the vertex, and $\phi^\star$ locally identifies $\{\de_0\}\t\R^{2n-k}$ with $S^k(X)$. Making $U^\star$ smaller if necessary we can suppose that $U^\star=U_1^\star\t U_2^\star$ for $\de_Q\in U_1^\star\subseteq X_Q$ and $U_2^\star\subseteq\R^{2n-k}$. Initially we define $W=\phi^\star(U^\star)$ and $V=\phi^\star\bigl(U^\star\cap(\{\de_0\}\t\R^{2n-k})\bigr)$. When we make $W,V$ smaller, we also make $U^\star,U_1^\star,U_2^\star$ smaller to maintain these relations.

Sometimes we need $V$ to be contractible, for example, $V$ could be a small open ball around $x$ in $\R^{2n-k}$. We make $W,V$ smaller if necessary to make~this~so.

By Theorem \ref{bc3thm2} for $X,x$, we can choose coordinates $(\th^\sk_1,\ldots,\th^\sk_k,x^\sk_{k+1},\ab\ldots,\ab x^\sk_n,\ab y^\sk_{k+1},\ab\ldots,y^\sk_n)$ on an open neighbourhood of $x$ in $S^k(X)$. Making $W,V$ above smaller if necessary, we take this open neighbourhood to be $V$, so that Theorem \ref{bc3thm2} gives global coordinates $(\th_1^\sk,\ldots,\th_k^\sk,x_{k+1}^\sk,\ab\ldots,\ab x_n^\sk,\ab y_{k+1}^\sk,\ab\ldots,y_n^\sk)$ on $V$ identifying $V$ with an open subset $\bar U_2^\star$ of $\R^{2n-k}$, and a basis of sections $v^\sk_1,\ldots,v^\sk_k,\frac{\ti\pd}{\pd\th^\sk_1},\ldots,\frac{\ti\pd}{\pd\th^\sk_k},\ab\frac{\ti\pd}{\pd x^\sk_{k+1}},\ab\ldots,\ab\frac{\ti\pd}{\pd x^\sk_n},\ab\frac{\ti\pd}{\pd y^\sk_{k+1}},\ab\ldots,\frac{\ti\pd}{\pd y^\sk_n}$ of ${}^bTX\vert_V$, which satisfy Theorem \ref{bc3thm2}(i)--(iv). Here we add superscripts $(\cdots)^\sk$ as above to denote objects defined on~$V$.

Now $U_2^\star\subset\R^{2n-k}$ is open, and $\phi^\star\vert_{\{\de_0\}\t U_2^\star}:U_2^\star\ra V$ is a diffeomorphism. Composing with coordinates $(\th^\sk_1,\ldots,y^\sk_n)$ gives a diffeomorphism $U_2^\star\ra\bar U_2^\star$. By conjugating by this diffeomorphism, we can change $U_2^\star$ and $\phi^\star$, replacing $U_2^\star$ by $\bar U_2^\star$, such that $\phi^\star\vert_{\{\de_0\}\t U_2^\star}:U_2^\star\ra V$ is the inverse of the coordinate map~$(\th^\sk_1,\ldots,y^\sk_n):V\ra U_2^\star$.

Define $(\th^\star_1,\ldots,\th^\star_k,x^\star_{k+1},\ldots,x^\star_n,y^\star_{k+1},\ldots,y^\star_n):W\ra\R^{2n-k}$ by the composition $W\,{\buildrel(\phi^\star)^{-1}\over\longra}\, U^\star=U_1^\star\t U_2^\star\,{\buildrel\Pi_{U_2^\star}\over\longra}\,U_2^\star\subset\R^{2n-k}$. Then $\th^\star_1,\ldots,y^\star_n$ are smooth real functions on $W$, whose restrictions to $V$ are~$\th_1^\sk,\ldots,y_n^\sk$.

As in Theorem \ref{bc3thm2}, we write $z_j^\sk=x_j^\sk+iy_j^\sk$ and $z_j^\star=x_j^\star+iy_j^\star$ for $j=k+1,\ldots,n$. Then $(\th^\sk_1,\ldots,\th^\sk_k,z^\sk_{k+1},\ldots,z^\sk_n)$ are coordinates on $V$ in~$\R^k\t\C^{n-k}$.

Theorem \ref{bc3thm2} gave a basis of sections $v_1,\ldots,v_k$ for ${}^bN_{S^k(X)}\vert_V,$ constant under $\nabla^{{}^bN_{S^k(X)}}$. Now $\phi^\star$ locally identifies $X_Q\t\R^{2n-k}$ with $X$, and this identifies ${}^bN_{S^k(X)}\vert_V$ with its flat connection with the trivial vector bundle with fibre $\Hom(Q^\gp,\R)$ with the trivial connection. As $V$ is contractible, this means that under the isomorphism ${}^bN_{S^k(X)}\vert_V\cong \Hom(Q^\gp,\R)$ induced by $\phi^\star$, $v_1,\ldots,v_k$ are identified with a basis of $\Hom(Q^\gp,\R)$, which we write as $\bar v_1,\ldots,\bar v_k$.

If $q\in Q\subset Q^\gp$, define a smooth function
\e
\th_q^\star:W\longra\R\quad\text{by}\quad \th_q^\star=\sum_{a=1}^k\bar v_a(q)\th_a^\star.
\label{bc3eq20}
\e 
Then $\th_{q+q'}^\star=\th_q^\star+\th_{q'}^\star$ for $q,q'\in Q$.

For each $q\in Q$ we have a smooth function $\la_q:X_Q\ra[0,\iy)$. Define $\mu_q^\star:W\ra[0,\iy)$ by $\mu_q^\star=\phi^\star_*(\la_q^\star\ci\Pi_{U_1^\star})$. Then $\mu_{q+q'}^\star=\mu_q^\star\cdot\mu_{q'}^\star$ for $q,q'\in Q$.

Now consider the following smooth complex-valued functions on $V$:
\e
\begin{split}
\mu^\star_qe^{i\th^\star_q}&:V\longra\C, \quad q\in Q, \quad\text{with $\mu_{q+q'}^\star e^{i\th^\star_{q+q'}}=\mu_q^\star e^{i\th^\star_q}\cdot\mu_{q'}^\star e^{i\th^\star_{q'}},$}\\
z^\star_j&:V\longra\C, \quad j=k+1,\ldots,n.
\end{split}
\label{bc3eq21}
\e
In the standard local model $\bigl((X_Q\t\R^{n-k})\t(\Hom(Q^\gp,\R)\t\R^{n-k}),J_\rst\bigr)$ in \eq{bc3eq7}, we can reorder the factors and write $(X_Q\t\Hom(Q^\gp,\R))\t(\R^{n-k}\t\R^{n-k})$. We can think of $\mu^\star_q$ as a function on $X_Q$ and $\th^\star_q$ as a function on $\Hom(Q^\gp,\R)$, so that $\mu^\star_qe^{i\th^\star_q}$ is a function on $X_Q\t\Hom(Q^\gp,\R)$, and think of $z^\star_{k+1},\ldots,z^\star_n$ as complex coordinates on~$\R^{n-k}\t\R^{n-k}\cong\C^{n-k}$.

As in \S\ref{bc224} there is a trivialization ${}^bTX_Q\cong\Hom(Q^\gp,\R)$ such that if $\al\in \Hom(Q^\gp,\R)$ and $q\in Q$ then $v_\al(\la_q)=\al(q)\cdot\la_q$, so that $(\la_q)^{-1}v_\al(\la_q)=v_\al(\log\la_q)=\al(q)$. Because of this, $(\la_q)^{-1}{}^b\d\la_q={}^b\d(\log\la_q)$ is a smooth section of ${}^bT^*X_Q$, which is identified with the constant section $q$ in the trivialization ${}^bT^*X_Q\cong\Hom(Q^\gp,\R)^*$. Note that this is true even where $\la_q=0$, so that $(\la_q)^{-1}$ does not make sense. Pushing forward by $\phi^*$, we see that $(\mu^\star_q)^{-1}{}^b\d\mu^\star_q={}^b\d(\log\mu^\star_q)$ is a smooth 1-form on $W$, which is well defined even where $\mu^\star_q=0$. Thus we can consider the $(0,1)$-forms
\begin{equation*}
{}^b\db(\log(\mu^\star_qe^{i\th^\star_q}))={}^b\db(\log\mu^\star_q)+i\th^\star_q\in\Ga^\iy({}^bT^{*\,0,1}W),
\end{equation*}
and these are well defined and smooth on $W$, even where $\mu^\star_q=0$. They satisfy ${}^b\db(\log\mu^\star_{q+q'}e^{i\th^\star_{q+q'}})={}^b\db(\log(\mu^\star_qe^{i\th^\star_q}))+{}^b\db(\log\mu^\star_{q'}e^{i\th^\star_{q'}})$ for~$q,q'\in Q$.

Write $\om^\sk_1,\ldots,\om^\sk_k,\ti\d\th^\sk_1,\ldots,\ti\d\th^\sk_k,\ab\ti\d x^\sk_{k+1},\ab\ldots,\ab\ti\d x^\sk_n,\ab\ti\d y^\sk_{k+1},\ab\ldots,\ti\d y^\sk_n$ for the basis of sections of ${}^bT^*W\vert_V$ dual to the basis $v^\sk_1,\ldots,\frac{\ti\pd}{\pd y^\sk_n}$ of ${}^bTW\vert_V$ above. Then $\ti\d\th^\sk_1,\ldots,\ti\d y^\sk_n$ are the images of $\d\th^\sk_1,\ldots,\d y^\sk_n$ under ${}^b\pi^*:T^*V\ra {}^bT^*W\vert_V$. By the dual of \eq{bc3eq6}, the action of $J$ on ${}^bT^*W\vert_V$ is given by
\e
J(\om^\sk_a)=-\ti\d\th^\sk_a,\quad  J(\ti\d\th^\sk_a)=\om^\sk_a,\quad
J(\ti\d x^\sk_j)=-\ti\d y^\sk_j, \quad
J(\ti\d y^\sk_j)=\ti\d x^\sk_j. 
\label{bc3eq22}
\e
It is sometimes convenient to use the notation
\e
\begin{aligned}
\ti\d\bar z^\sk_j&=\ti\d\th^\sk_j+i\om^\sk_a, & j&=1,\ldots,k,\\
\ti\d\bar z^\sk_j&=\ti\d x^\sk_j-i\ti\d y^\sk_j, & j&=k+1,\ldots,n.
\end{aligned}
\label{bc3eq23}
\e
Then $\ti\d\bar z^\sk_1,\ldots,\ti\d\bar z^\sk_n$ is a basis of sections of ${}^bT^{*\,0,1}W\vert_V$, by \eq{bc3eq22}.

From \eq{bc3eq22} we see that
\e
{}^b\db z^\star_j\vert_V=\ha({}^b\d+iJ\ci{}^b\d)(x_j+iy_j)\vert_V=\ha(\id+iJ)(\ti\d x^\sk_j+i\ti\d y^\sk_j)=0.
\label{bc3eq24}
\e
We can show that for $q\in Q$,
\e
{}^b\d(\log\mu^\star_q)\vert_V=\sum_{a=1}^k\bar v_a(q)\om^\sk_a.
\label{bc3eq25}
\e
Then by \eq{bc3eq20}, \eq{bc3eq22} and \eq{bc3eq25} we have
\e
{}^b\db(\log(\mu^\star_qe^{i\th^\star_q}))\vert_V=\ha(\id+iJ)\Bigl(\sum_{a=1}^k\bar v_a(q)
(\om^\sk_a+i\ti\d\th^\sk_a)\Bigr)=0.
\label{bc3eq26}
\e

Since $\d(\log\mu^\star_q)\in\Ga^\iy({}^bT^*W)$ with $\d(\log\mu^\star_{q+q'})=\d(\log\mu^\star_q)+\d(\log\mu^\star_{q'})$ for $q,q'\in Q$, as for \eq{bc3eq25} there exist unique $\om_1^\star,\ldots,\om_k^\star\in\Ga^\iy({}^bT^*W)$ such that $\om_a^\star\vert_V=\om^\sk_a$ and for all $q\in Q$, we have
\begin{equation*}
{}^b\d(\log\mu^\star_q)=\sum_{a=1}^k\bar v_a(q)\om^\star_a.
\end{equation*}
As for the first part of \eq{bc3eq26}, for $q\in Q$ we have
\e
{}^b\db(\log(\mu^\star_qe^{i\th^\star_q}))=\ha(\id+iJ)\Bigl(\sum_{a=1}^k\bar v_a(q)
(\om^\star_a+i\ti\d\th^\star_a)\Bigr).
\label{bc3eq27}
\e

Write $C^\iy(W,\C)$ for the $\C$-algebra of smooth functions $f:W\ra\C$. Write $I\subset C^\iy(W,\C)$ for the ideal of $f\in C^\iy(W,\C)$ with $f\vert_V=0$. Then we can consider the $N^{\rm th}$ power ideal $I^N\subset C^\iy(W,\C)$ for $N=0,1,2,\ldots,$ and the ideal $I^\iy=\bigcap_{N\ge 0}I^N$ of functions on $W$ vanishing to infinite order at $V$. As $\Ga^\iy({}^bT^*W\ot\C)$, $\Ga^\iy({}^bT^{*\,1,0}W)$, $\Ga^\iy({}^bT^{*\,0,1}W)$ are modules over $C^\iy(W,\C)$, we can also consider the subspaces $I^N\cdot \Ga^\iy({}^bT^*W\ot\C)$, $I^N\cdot\Ga^\iy({}^bT^{*\,1,0}W)$, $I^N\cdot\Ga^\iy({}^bT^{*\,0,1}W)$ of 1-forms, and $(1,0)$-forms, and $(0,1)$-forms, on $W$, which vanish to order $N=0,1,2,\ldots,\iy$ at $V$. Here by $I^N\cdot \Ga^\iy({}^bT^*W\ot\C)$ we mean the vector subspace of elements of $\Ga^\iy({}^bT^*W\ot\C)$ which may be written as a finite sum $\sum_{j=1}^mf_j\al_j$ for $f_j\in I^N$ and~$\al_j\in\Ga^\iy({}^bT^*W\ot\C)$.

From \eq{bc3eq24} and \eq{bc3eq26} we deduce that
\ea
{}^b\db z^\star_j&\in I^1\cdot \Ga^\iy({}^bT^{*\,0,1}W),\quad j=k+1,\ldots,n.
\label{bc3eq28}\\
{}^b\db(\log(\mu^\star_qe^{i\th^\star_q}))&\in I^1\cdot \Ga^\iy({}^bT^{*\,0,1}W),\quad q\in Q.
\label{bc3eq29}
\ea

For each $N\in\N$, define an ideal $Q_N\subseteq Q$ in $Q$ by $Q_0=Q$ and for $N>0$
\e
Q_N=\bigl\{q_1+\cdots+q_N:q_1,\ldots ,q_N \in Q \sm\{0\}\bigr\}.
\label{bc3eq30}
\e
Then $Q_0\supseteq Q_1\supseteq\cdots$ with $\bigcap_{N\ge 0}Q_N=\es$. As $Q$ is toric, by Cox--Little--Schenk \cite[Lem.~1.2.23a]{CLS} $Q_N\sm Q_{N+1}$ is a finite set for all $N\ge 0$. Clearly, if $q\in Q_1=Q\sm\{0\}$ then $\mu^\star_q\in I$, and if $q\in Q_N$ then $\mu^\star_q\in I^N$ for all~$N\ge 0$.

The g-chart $(P,U^\star,\phi^\star)$ identifies $J$ and $J_\rst$ in Theorem \ref{bc3thm3} if and only if the functions in \eq{bc3eq21} are holomorphic with respect to $J$ on $V$, that is, if ${}^b\db(\mu^\star_qe^{i\th^\star_q})={}^b\db z^\star_j=0$ for all $q,j$. Also $(P,U^\star,\phi^\star)$ satisfies the conditions of Theorem \ref{bc3thm3} if the functions in \eq{bc3eq21} are formally holomorphic to infinite order along $V$, that is, if ${}^b\db(\mu^\star_qe^{i\th^\star_q}),{}^b\db z^\star_j\in I^\iy\cdot\Ga^\iy({}^bT^{*\,0,1}W)$ for all~$q,j$.

The idea of the proof is to make successive corrections to the $\mu^\star_qe^{i\th^\star_q},z^\star_j$ so that ${}^b\db(\mu^\star_qe^{i\th^\star_q}),{}^b\db z^\star_j$ lie in $I^N\cdot\Ga^\iy({}^bT^{*\,0,1}W)$ for increasing $N=1,2,\ldots,$ by induction on $N$, starting with \eq{bc3eq28}--\eq{bc3eq29} when $N=1$, and then by taking a limit $N\ra\iy$ and using the Borel Theorem, to pass to~$N=\iy$.
\end{dfn}

\subsubsection{Describing the ideals $I^N$ in $C^\iy(W,\C)$}
\label{bc342}

Continue in the situation of Definition \ref{bc3def2}. We will study the ideals $I^N$ in $C^\iy(W,\C)$ for $N=0,1,2,\ldots,\iy$.

\begin{lem}
\label{bc3lem2}
Let\/ $f \in C^\iy(W,\C)$ and\/ $N \in \Z_{> 0}$. Then, $f \in I^N$ if and only if there exist\/ $q_1,\ldots,q_l\in Q_N$ and\/ $f_1,\ldots,f_l\in C^\iy(W,\C)$ with\/~$f=\sum_{i=1}^l f_i\mu^\star_{q_i}$.
\end{lem}

\begin{proof}
As in Definition \ref{bc3def2} we have $\mu^\star_{q_1} \cdots \mu^\star_{q_N} = \mu^\star_{q_1+\cdots+q_N}$. Thus, the lemma for $N=1$ implies the lemma for general $N$. We will prove the lemma for~$N=1$.

For $N=1$, the `only if' part is immediate, since if $f= \sum_{i=1}^l f_i\mu^\star_{q_i}$ for $q_i\in Q_1=Q\sm\{0\}$ then $f\vert_V= 0$ since $\mu^\star_{q_i}\vert_V=0$.

For the `if' part with $N=1$, let $f\in C^\iy(W,\C)$. By the characterization of smooth functions, as in \eq{bc2eq6}, there exists a smooth function $g\in C^\iy(Y,\C)$ for $Y\subset [0,\iy)^l\t \R^k\t\C^{n-k}$ open, and $q_1,\ldots,q_l\in Q_1=Q\sm\{ 0 \}$ such that
\e
f(x)=g(\mu^\star_{q_1}(x),\ldots, \mu^\star_{q_l(x)},\th_1^\star,\ldots,\th_k^\star,z_{k+1}^\star,\ldots,z_n^\star),
\label{bc3eq31}
\e
where the entries of $g$ map $W\!\ra\! Y$. As $f\vert_V\!=\!0$, we have $g\vert_{Y\cap (\{(0,\ldots,0)\}\t\R^k\t\C^{n-k})}\ab =0$. By applying Hadamard's Lemma to $g\in C^\iy(W,\C)$, for $(r_1,\ldots, r_l)\in[0,\iy)^l$ and $r' \in \R^k\t\C^{n-k}$, we get
\begin{equation*} 
g(r_1,\ldots,r_l,r')=\sum_{i=1}^lg_i(r_1,\ldots,r_l,r')r_i. 
\end{equation*}
So, substituting $r_i=\mu^\star_{q_i}$ and $r'=(\th_1^\star,\ldots,z_n^\star)$ and using \eq{bc3eq31}, we obtain
\begin{equation*} 
f(x)=\sum_{i=1}^l  g_i(\mu^\star_{q_1}(x),\ldots,\mu^\star_{q_l(x)},\th_1^\star,\ldots,\th_k^\star,z_{k+1}^\star,\ldots,z_n^\star)\mu^\star_{q_i}(x), 
\end{equation*}
The result then follows by setting $f_i=g_i(\mu^\star_{q_1}(x),\ldots ,\mu^\star_{q_l(x)},\th_1^\star,\ldots,z_n^\star)$.
\end{proof}

The following lemma describes the successive quotients $I^N/I^{N+1}$ explicitly.

\begin{lem}
\label{bc3lem3}
For any $N\in\N$, the following holds:
\begin{itemize}
\setlength{\itemsep}{0pt}
\setlength{\parsep}{0pt}
\item[{\rm(i)}] There is an isomorphism
\e
I^N/I^{N+1}\cong\bigop_{q\in Q_N\sm Q_{N+1}}C^\iy(V,\C)\mu^\star_q,
\label{bc3eq32}
\e
of free $C^\iy(V,\C)$-modules of finite rank, where $Q_N\sm Q_{N+1}$ is finite.
\item[{\rm(ii)}] Given $f_1,\ldots,f_l\in C^\iy(W,\C)$ and distinct\/ $q_1,\ab \ldots ,\ab q_l\ab \in Q_N \sm Q_{N+1}$ such that\/ $\sum_{i=1}^l \mu^\star_{q_i} f_i \in I^{N+1},$ then $f_1,\ldots, f_l \in I$.
\end{itemize}
\end{lem}

\begin{proof}
For (i), $Q_N\sm Q_{N+1}$ is finite as in Definition \ref{bc3def2}, and as $C^\iy(W,\C)/I\cong C^\iy(V,\C)$ both sides of \eq{bc3eq32} are $C^\iy(V,\C)$-modules, so the right hand side of \eq{bc3eq32} is a finite rank $C^\iy(V,\C)$-module. Write
\begin{equation*} 
Q_{1}\sm Q_{2} =  \{  q_1,\ldots , q_m \} \,. 
\end{equation*}

As in Definition \ref{bc3def2}, $\phi^\star:U_1^\star\t U_2^\star\ra W$ is a diffeomorphism, where $\de_0\in U_1^\star\subseteq X_Q$ and $U_2^\star\subset\R^{2n-k}$ are open. Apply Proposition \ref{bc2prop1} to $Q$ with generators $q_1,\ldots,q_m$ and a generating set of relations for $q_1,\ldots,q_m$ of the form
\begin{equation*}
a_1^jq_1+\cdots+a_m^jq_m=b_1^jq_1+\cdots+b_m^jq_m\quad\text{in $Q$ for $j=1,\ldots,l,$}
\end{equation*}
where $a_i^j,b_i^j\in\N$ for $i=1,\ldots,m$ and $j=1,\ldots,l$. Then $\la_{q_1}\t\cdots\t\la_{q_m}:X_Q\hookra [0,\iy)^m$ identifies $X_Q$ with $X_Q'\subset[0,\iy)^m$ in \eq{bc2eq7}. Pick an open set $Y\subseteq[0,\iy)^m$ such that $Y\cap X_Q'=\la_{q_1}\t\cdots\t\la_{q_m}(U_1^\star)$. Then
\e
C^\iy(W,\C)\cong C^\iy(Y\t U_2^\star)/K,
\label{bc3eq33}
\e
where $K$ is the ideal in $C^\iy(Y\t U_2^\star,\C)$
\begin{equation*}
K=\ts\bigl(\prod_{i=1}^m x_i^{a_i^j}-\prod_{j=1}^mx_i^{b_i^j}:j=1,\ldots,l\bigr),
\end{equation*}
with $(x_1,\ldots,x_m)$ coordinates on $Y\subset[0,\iy)^m$. Write $\ti I\subset C^\iy(Y\t U_2^\star,\C)$ for the ideal of functions vanishing on $\{(0,\ldots,0)\}\t U_2^\star$. Then the isomorphism \eq{bc3eq33} identifies $I\cong\ti I/K$, and so it identifies $I^N\cong(\ti I^N+K)/K$, and 
\begin{equation*}
I^N/I^{N+1}\cong\frac{(\ti I^N+K )/K}{(\ti I^{N+1}+K)/K}\cong\frac{\ti I^N+K}{\ti I^{N+1}+K} \,\ \cong\frac{\ti I^N}{\ti I^{N+1}+(K\cap\ti I^N)}.  
\end{equation*}
So, we have a surjective map
\e
\varphi:\ti I^N/\ti I^{N+1}\longra\ti I^N/(\ti I^{N+1}+(K\cap\ti I^N))\cong I^N/I^{N+1}.
\label{bc3eq34}
\e

As smooth functions on $Y\subseteq [0,\iy)^m$ have Taylor expansions at $(0,\ldots,0)$, we can identify the quotient $\ti I^N/\ti I^{N+1}$ with the set of homogeneous polynomials in $x_1,\ldots,x_m$ of total degree $N$ with coefficients in $C^\iy(V,\C)$, as $V\cong U_2^\star$
\e 
\ti I^N/\ti I^{N+1}\cong\bigop_{(a_1,\ldots , a_m)\in\N^m : \sum_{i=1}^m a_i = N} C^\iy(V,\C)\cdot \prod_{i=1}^m x_i^{a_i}. 
\label{bc3eq35}
\e
With this identification, the map $\varphi:\ti I^N/\ti I^{N+1}\twoheadrightarrow I^N/I^{N+1}$ is defined by
\begin{equation*}  
\varphi\bigl(\ts\prod_{i=1}^mx_i^{a_i}\bigr) = \mu^\star_{\sum_{i=1}^m a_iq_i}+I^{N+1}. 
\end{equation*}

We can describe the kernel of $\varphi$ in \eq{bc3eq34} in terms of the presentation \eq{bc3eq35} as follows. Firstly, if $(a_1,\ldots,a_m)\in\N^m$ with $\sum_{i=1}^ma_i=N$ and $\sum_{i=1}^m a_iq_i\in Q_{N+1}$ then $\vp(\prod_{i=1}^m x_i^{a_i})=0$. Secondly, if $(a_1,\ldots,a_m),(a'_1,\ldots,a'_m)\in\N^m$ with $\sum_{i=1}^ma_i=\sum_{i=1}^ma'_i=N$ and $\sum_{i=1}^m a_iq_i=\sum_{i=1}^m a'_iq_i\in Q_N\sm Q_{N+1}$ then $\vp(\prod_{i=1}^m x_i^{a_i})=\vp(\prod_{i=1}^m x_i^{a'_i})$. Thus, quotienting \eq{bc3eq35} by the kernel of $\vp$ is equivalent to deleting those $(a_1,\ldots,a_m)$ in the sum with $\sum_{i=1}^m a_iq_i\in Q_{N+1}$, and identifying those $(a_1,\ldots,a_m),(a'_1,\ldots,a'_m)\in\N^m$ in the sum with $q=\sum_{i=1}^m a_iq_i=\sum_{i=1}^m a'_iq_i\in Q_N\sm Q_{N+1}$, and then we can replace $\prod_{i=1}^m x_i^{a_i}$ by $\mu^\star_q$. Part (i) of the lemma follows.

For part (ii), given $f_1,\ldots,f_l\in C^\iy(W,\C)$ and distinct $q_1,\ab \ldots ,\ab q_l\ab \in Q_N \sm Q_{N+1}$ with $\sum_{i=1}^l \mu^\star_{q_i} f_i \in I^{N+1},$ we have $\sum_{i=1}^l \mu^\star_{q_i} f_i=0$ in $I^N/I^{N+1}$. By (i), $\mu^\star_{q_i}$ for $q_i\in Q_{ N}\sm Q_{ N+1}$ are linearly independent over $C^\iy(V,\C)$ in $I^N/I^{N+1}$, so $f_i\vert_V=0$, that is, $f_i\in I$, for all $1 \leq i\leq l$. This completes the proof.
\end{proof}

Here is a version of Borel's Theorem for manifolds with g-corners.

\begin{prop}
\label{bc3prop1}
For every formal power series 
\e
T=\sum_{q \in Q} g_q (\th^\star_1, \dots, \th^\star_k, z_{k+1}^\star, \dots, z_n^\star)   \mu^\star_q \,,
\label{bc3eq36}
\e
where the $g_q$ are smooth functions on $V\cong U_2^\star$ pulled back to $W\cong U_1^\star\t U_2^\star,$ there exists a smooth complex function $f$ on $W$ such that\/ $T$ is the Taylor series of\/ $f$ along $V$. That is, $f+I^N=T+I^N$ in $C^\iy(W,\C)/I^N$ for all\/~$N\ge 0$.
\end{prop}

\begin{proof}
In the proof of Lemma \ref{bc3lem3} we explained how to relate smooth functions such as $g(\th^\star_1,\ldots,z_n^\star)\cdot\mu_{\sum_{i=1}^m a_iq_i}^\star$ on $W\cong U_1^\star\t U_2^\star$ to smooth functions on $Y\t U_2^\star$ such as $g(\th^\star_1,\ldots,z_n^\star)\cdot\prod_{i=1}^mx_i^{a_i}$, for open $(0,\ldots,0)\in Y\subset[0,\iy)^m$. By choosing an arbitrary right inverse for the map $\N^m\twoheadrightarrow Q$, we can lift power series on $W$ along $V$ such as $T$ in \eq{bc3eq36} to power series on $Y\t U_2^\star$ along $\{(0,\ldots,0)\}\t U_2^\star$. We can then apply Borel's Theorem (see for example Moerdijk and Reyes \cite[Th.~1.3]{MoRe}) to get a smooth function on $Y\t U_2^\star$ with this power series, and pull this back to $W\cong U_1^\star\t U_2^\star$ to deduce the proposition.
\end{proof}

\subsubsection{The ${}^b\db_{S^k}$-operator on $V\subset S^k(X)$}
\label{bc343}

As part of the proof of Theorem \ref{bc3thm3}, we will want to modify $f\in C^\iy(W,\C)$ which is b-holomorphic modulo $I^N$ to $\ti f\in C^\iy(W,\C)$ which is b-holomorphic modulo $I^{N+1}$. Using the results of \S\ref{bc342}, we will reduce this to solving a p.d.e.\ on $V$. This p.d.e.\ involves an operator ${}^b\db_{S^k}$ which we study next.

\begin{dfn}
\label{bc3def3}
Continue in the situation of Definition \ref{bc3def2}. As in \S\ref{bc214} and \S\ref{bc226}, we have a Lie algebroid ${}^b\pi\vert_V:{}^bTW\vert_V\ra TV$ on $V$ which maps the Lie bracket $[\,,\,]_{S^k}$ on $\Ga^\iy({}^bTW\vert_V)$ to the usual Lie bracket on $\Ga^\iy(TV)$. 

By general properties of Lie algebroids, we have an analogue of de Rham forms: there are differentials
\e
{}^b\d_{S^k}:\Ga^\iy\bigl(\La^r{}^bT^*W\vert_V\bigr)\longra \Ga^\iy\bigl(\La^{r+1}{}^bT^*W\vert_V\bigr)
\label{bc3eq37}
\e
for $0\le r<2n$, with $({}^b\d_{S^k})^2=0$. Since $[\,,\,]_{S^k}$ is defined by restriction of the Lie bracket on $\Ga^\iy({}^bTW)$ to $V$, these can be characterized as follows: if $\ti\al\in \Ga^\iy(\La^r{}^bT^*W)$ with $\ti\al\vert_V=\al$, then ${}^b\d_{S^k}\al=({}^b\d\ti\al)\vert_V$.

We complexify \eq{bc3eq37} to give 
\begin{equation*}
{}^b\d_{S^k}:\Ga^\iy\bigl(\La^r{}^bT^*W\vert_V\ot_\R\C\bigr)\longra \Ga^\iy\bigl(\La^{r+1}{}^bT^*W\vert_V\ot_\R\C\bigr).
\end{equation*}
Then we use the splitting ${}^bT^*W\ot_\R\C={}^bT^{*\,1,0}W\op {}^bT^{*\,0,1}W$ from the complex structure $J$ on $W$ to decompose $\La^r{}^bT^*W\vert_V\ot_\R\C$ into $(p,q)$-forms for $p+q=r$, and to decompose ${}^b\d_{S^k}$ into $\pd$ and $\db$-operators ${}^b\pd_{S^k},{}^b\db_{S^k}$ in the usual way.

In particular, for $0\le r<n$ we obtain an operator
\begin{equation*}
{}^b\db_{S^k}:\Ga^\iy\bigl(\La^r{}^bT^{*\,0,1}W\vert_V\bigr)\longra\Ga^\iy\bigl(\La^{r+1}T^{*\,0,1}W\vert_V\bigr), 
\end{equation*}
such that for all $\al \in \Ga^\iy (\La^rT^{*\,0,1}W\vert_V)$, and for all $\ti\al \in \Ga^\iy (\La^r T^{*\,0,1} W )$ with $\ti\al\vert_V=\al $ we have ${}^b\db_{S^k}\al={}^b \db \ti\al\vert_V$, where ${}^b\db$ is as in Definition~\ref{bc3def1}.

Using the notation $\ti\d\bar z^\sk_j$ from \eq{bc3eq23}, a basis of sections of $\La^r{}^bT^{*\,0,1}W\vert_V$ is $\bigl(\ti\d\bar z^\sk_{j_1}\w\cdots\w\ti\d\bar z^\sk_{j_r}\bigr)_{1\le j_1<j_2<\cdots<j_r\le n}$. Using this, we can write ${}^b\db_{S^k}$ explicitly~as
\ea
&{}^b\db_{S^k}\Bigl(\sum_{1\le j_1<j_2<\cdots<j_r\le n}\al_{j_1j_2\cdots j_r}\ti\d\bar z^\sk_{j_1}\w\cdots\w\ti\d\bar z^\sk_{j_r}\Bigr)
\label{bc3eq38}\\
&=\ha\sum_{\begin{subarray}{l} 1\le j_1<j_2<\cdots<j_r\le n,\\
a=1,\ldots,k\end{subarray}}\frac{\pd\al_{j_1j_2\cdots j_r}}{\pd\th_a^\sk}\,\ti\d\bar z^\sk_a\w
\ti\d\bar z^\sk_{j_1}\w\cdots\w\ti\d\bar z^\sk_{j_r}
\nonumber\\
&+\ha\sum_{\begin{subarray}{l} 1\le j_1<j_2<\cdots<j_r\le n,\\
a=k+1,\ldots,n\end{subarray}}\Bigl(\frac{\pd\al_{j_1j_2\cdots j_r}}{\pd x_a^\sk}-i\frac{\pd\al_{j_1j_2\cdots j_r}}{\pd y_a^\sk}\Bigr)\,\ti\d\bar z^\sk_a\w
\ti\d\bar z^\sk_{j_1}\w\cdots\w\ti\d\bar z^\sk_{j_r},
\nonumber
\ea
for smooth functions $\al_{j_1j_2\cdots j_r}:V\ra\C$.
\end{dfn}

For b-complex manifolds with boundary, a ${}^b\db_{S^1}$-Poincar\'e Lemma is stated in Mendoza \cite[Lem.~5.3]{Mend2}. We have the following analogue of this for b-complex manifolds with g-corners. It follows from material in Tr\`eves~\cite[\S VI.7]{Trev}.

\begin{prop}
\label{bc3prop2}
Noting that\/ $V$ is assumed to be contractible, the following complex is exact:
\begin{equation*}
\text{\begin{footnotesize}$\displaystyle
\xymatrix@C=8.5pt{
0 \ar[r] & \C \ar[r]^(0.19)1 & \Ga^\iy \bigl(\La^0T^{*\,0,1}W\vert_V\bigr) \ar[r]^{{}^b\db_{S^k}} &
\Ga^\iy \bigl(\La^1T^{*\,0,1}W\vert_V\bigr) \ar[r]^(0.75){{}^b\db_{S^k}} & \cdots \ar[r]^(0.22){{}^b\db_{S^k}} & \Ga^\iy \bigl(\La^nT^{*\,0,1}W\vert_V\bigr) \ar[r] & 0. }
$\end{footnotesize}}
\end{equation*}
\end{prop}

\subsubsection{Holomorphic functions on the completion of $W$ along $V$}
\label{bc344}

\begin{prop}
\label{bc3prop3}
{\bf(a)} For each\/ $j=1,\ldots,k,$ there exists a family $(g_j^q)_{q\in Q_1}$ in $C^\iy(W,\C)$ satisfying for all\/ $N\ge 1$
\e
\ha(\id+iJ)(\om^\star_j+i{}^b\d\th^\star_j)+\sum_{q\in Q_1\sm Q_N}{}^b\db \bigl(g_j^q\mu^\star_qe^{i\th^\star_q}\bigr)
\in I^N\cdot \Ga^\iy({}^bT^{*\,0,1}W).
\label{bc3eq39}
\e

\noindent{\bf(b)} For each\/ $j=k+1,\ldots,n,$ there exists a family $(h_j^q)_{q\in Q_1}$ in $C^\iy(W,\C)$ satisfying for all\/ $N\ge 1$
\e
{}^b\db z_j^\star+\sum_{q\in Q_1\sm Q_N}{}^b\db \bigl(h_j^q\mu^\star_qe^{i\th^\star_q}\bigr)
\in I^N\cdot \Ga^\iy({}^bT^{*\,0,1}W).
\label{bc3eq40}
\e
\end{prop}

\begin{proof}
We prove (a),(b) by induction, with inductive hypothesis for $M=1,2,\ldots$
\begin{itemize}
\setlength{\itemsep}{0pt}
\setlength{\parsep}{0pt}
\item[$(*)_M$] We have chosen $g_j^q,h_j^q\in C^\iy(W,\C)$ for all $q\in Q_1\sm Q_M$ such that \eq{bc3eq39}--\eq{bc3eq40} hold for all $N=1,2,\ldots,M$.
\end{itemize}
For the first step $M=1$, there are no $g_j^q,h_j^q$ to choose. Equation \eq{bc3eq30} for $N=1$ holds as by \eq{bc3eq22}
\begin{equation*}
\ha(\id+iJ)(\om^\star_j+i{}^b\d\th^\star_j)\vert_V=\ha(\id+iJ)(\om^\sk_j+i\ti\th^\sk_j)=0.
\end{equation*}
Equation \eq{bc3eq40} for $N=1$ holds by~\eq{bc3eq28}.

For the inductive step, suppose that $(*)_M$ holds for some $M\ge 1$. Then for (a) we must choose $g_j^q\in C^\iy(W,\C)$ for all $q\in Q_M\sm Q_{M+1}$ such that in $\Ga^\iy({}^bT^{*\,0,1}W)/I^{M+1}\Ga^\iy({}^bT^{*\,0,1}W)$ we have
\ea
&\sum_{q\in Q_M\sm Q_{M+1}}({}^b\db g_j^q)\cdot\mu^\star_qe^{i\th^\star_q} +I^{M+1}\Ga^\iy({}^bT^{*\,0,1}W)
\label{bc3eq41}\\
&=\sum_{q\in Q_M\sm Q_{M+1}}{}^b\db \bigl(g_j^q\mu^\star_qe^{i\th^\star_q}\bigr)
+I^{M+1}\Ga^\iy({}^bT^{*\,0,1}W)
\nonumber\\
&=-\ha(\id+iJ)(\om^\star_j+i{}^b\d\th^\star_j)-\sum_{q\in Q_1\sm Q_M}{}^b\db \bigl(g_j^q\mu^\star_qe^{i\th^\star_q}\bigr)
+I^{M+1}\Ga^\iy({}^bT^{*\,0,1}W).
\nonumber
\ea
Here the first two lines are equal because if $q\in Q_M\sm Q_{M+1}$ then ${}^b\db \bigl(\mu^\star_qe^{i\th^\star_q}\bigr)\in I^{M+1}\Ga^\iy({}^bT^{*\,0,1}W)$ by $\mu^\star_qe^{i\th^\star_q}\in I^M$ and \eq{bc3eq29}, and equality of the second and third lines is equivalent to \eq{bc3eq39} for $N=M+1$.

Now the top line of \eq{bc3eq41} lies in $I^M\Ga^\iy({}^bT^{*\,0,1}W)/I^{M+1}\Ga^\iy({}^bT^{*\,0,1}W)$ as $\mu^\star_qe^{i\th^\star_q}\in I^M$ for $q\in Q_M\sm Q_{M+1}$, and the bottom line lies in this by $(*_M)$. Thus we can regard \eq{bc3eq41} as an equation in~$I^M\Ga^\iy({}^bT^{*\,0,1}W)/I^{M+1}\Ga^\iy({}^bT^{*\,0,1}W)$.

As ${}^bT^{*\,0,1}W$ is a trivial vector bundle, Lemma \ref{bc3lem3} implies that
\e
\frac{I^M\Ga^\iy({}^bT^{*\,0,1}W)}{I^{M+1}\Ga^\iy({}^bT^{*\,0,1}W)}
\cong\bigop_{q\in Q_M\sm Q_{M+1}}\Ga^\iy\bigr({}^bT^{*\,0,1}W\vert_V\bigr)\mu^\star_qe^{i\th^\star_q}.
\label{bc3eq42}
\e
Under this identification, the left hand side of \eq{bc3eq41} is identified with 
\e
\sum_{q\in Q_M\sm Q_{M+1}}\bigl({}^b\db_{S^k}(g_j^q\vert_V)\bigr)\cdot\mu^\star_qe^{i\th^\star_q},
\label{bc3eq43}
\e
for ${}^b\db_{S^k}$ as in \S\ref{bc343}. As ${}^b\db^2=0$, we see from \eq{bc3eq27} that the bottom line of \eq{bc3eq41}, omitting $I^{M+1}\cdots$, lies in the kernel of ${}^b\db$. Thus under the isomorphism \eq{bc3eq42}, the bottom line of \eq{bc3eq41} is identified with an element of the kernel of ${}^b\db_{S^k}$ acting on the $\Ga^\iy\bigr({}^bT^{*\,0,1}W\vert_V\bigr)$ factors, noting \eq{bc3eq29}. But by Proposition \ref{bc3prop2} the kernel of ${}^b\db_{S^k}$ equals the image of ${}^b\db_{S^k}$ in $\Ga^\iy\bigr({}^bT^{*\,0,1}W\vert_V\bigr)$. 

Thus there exist functions $g_j^q\vert_V\in C^\iy(V,\C)$ for all $q\in Q_M\sm Q_{M+1}$, unique up to addition of constants, such that \eq{bc3eq43} is identified with the bottom line of \eq{bc3eq41} under \eq{bc3eq42}. We then choose arbitrary $g_j^q\in C^\iy(W,\C)$ with the given restrictions $g_j^q\vert_V\in C^\iy(V,\C)$ for all $q\in Q_M\sm Q_{M+1}$, and \eq{bc3eq41} holds. This proves \eq{bc3eq39} for $N=M+1$. The proof of \eq{bc3eq40} for $N=M+1$ is essentially the same. Thus $(*)_{M+1}$ holds, and the proposition follows by induction.
\end{proof}

Combining Propositions \ref{bc3prop1} and \ref{bc3prop3} yields

\begin{cor}
\label{bc3cor1}
In the situation of Proposition\/ {\rm\ref{bc3prop3},} there exist $g_j\in I$ for $j=1,\ldots,k$ and\/ $h_j\in I$ for $j=k+1,\ldots,n$ such that
\e
\begin{split}
g_j-\sum_{q\in Q_1\sm Q_N}g_j^q\mu^\star_qe^{i\th^\star_q}&\in I^{N+1}, \quad N\ge 1,\;\> j=1,\ldots,k,\\
h_j-\sum_{q\in Q_1\sm Q_N}h_j^q\mu^\star_qe^{i\th^\star_q}&\in I^{N+1}, \quad  N\ge 1,\;\>  j=k+1,\ldots,n.
\end{split}
\label{bc3eq44}
\e
Equations\/ {\rm\eq{bc3eq39}--\eq{bc3eq40}} and\/ \eq{bc3eq44} then imply that
\e
\begin{split}
\ha(\id\!+\!iJ)(\om^\star_j\!+\!i{}^b\d\th^\star_j)\!+\!{}^b\db g_j&\in I^\iy\cdot \Ga^\iy({}^bT^{*\,0,1}W), 
\;\> j=1,\ldots,k,\\
{}^b\db z_j^\star+{}^b\db h_j&\in I^\iy\cdot \Ga^\iy({}^bT^{*\,0,1}W), \;\>  j=k\!+\!1,\ldots,n.
\end{split}
\label{bc3eq45}
\e
\end{cor}

\subsubsection{Finishing the proof}
\label{bc345}

Use the notation of \S\ref{bc341} and Corollary \ref{bc3cor1}. For $q\in Q$ and $j=k+1,\ldots,n$ define smooth functions $\mu_q:W\ra[0,\iy)$, $\th_q:W\ra\R$ and $z_j:W\ra\C$ by
\e
\mu_q=\mu_q^\star\cdot e^{\sum_{a=1}^k\bar v_a(q)\Re g_a},\quad
\th_q=\th_q^\star+\sum_{a=1}^k\bar v_a(q)\Im g_a, \quad z_j=z_j^\star+h_j.
\label{bc3eq46}
\e
Then $\mu_{q+q'}=\mu_q\cdot\mu_{q'}$ and $\th_{q+q'}=\th_q+\th_{q'}$ for all $q,q'\in Q$. We have $\mu_qe^{i\th}=\mu^\star_qe^{i\th^\star_q}\cdot e^{\sum_{a=1}^k\bar v_a(q)g_a}$.

From \eq{bc3eq27}, \eq{bc3eq45} and \eq{bc3eq46} we deduce that for $q\in Q$ and $j=k+1,\ldots,n$ 
\e
{}^b\db(\mu_qe^{i\th_q}),{}^b\db z_j\in I^\iy\cdot \Ga^\iy({}^bT^{*\,0,1}W).
\label{bc3eq47}
\e
That is, $\mu_qe^{i\th_q}$ and $z_j$ are `$J$-holomorphic to infinite order in $I$'. We regard $\mu_q,\th_q,z_j$ as `corrected' versions of $\mu^\star_q,\th^\star_q,z^\star_j$.

Recall that $(\phi^\star)^{-1}:W\ra X_Q\t(\R^k\t\C^{n-k})$ is a diffeomorphism with the open subset $U_1^\star\t U_2^\star\subset X_Q\t(\R^k\t\C^{n-k})$, where the $U_1^\star$-component $W\ra X_Q$ of $(\phi^\star)^{-1}$ is determined by the maps $\mu^\star_q:W\ra[0,\iy)$ for $q\in Q$ with $\mu^\star_{q+q'}=\mu^\star_q\cdot\mu^\star_{q'}$, and the $U_2^\star$-component $W\ra\R^k\t\C^{n-k}$ of $(\phi^\star)^{-1}$ is $(\th_1^\star,\ldots,\th_k^\star,z_{k+1}^\star,\ldots,z_n^\star)$. The corrected versions $\mu_q,\th_q,z_j$ also define a smooth map $\phi^{-1}:W\ra X_Q\t(\R^k\t\C^{n-k})$ given on the $X_Q$ component by the maps $\mu_q:W\ra[0,\iy)$ for $q\in Q$ with $\mu_{q+q'}=\mu_q\cdot\mu_{q'}$ and on the $\R^k\t\C^{n-k}$ component by~$(\th_1,\ldots,\th_k,z_{k+1},\ldots,z_n)$. 

We think of $\phi^{-1}$ as a deformation of $(\phi^\star)^{-1}$ which agrees with $(\phi^\star)^{-1}$ to leading order at $V\subset W$, since $g_a\vert_V=h_j\vert_V=0$, so in particular $\phi^{-1}\vert_V=(\phi^\star)^{-1}\vert_V$ and ${}^bT(\phi^{-1})\vert_V={}^bT((\phi^\star)^{-1})\vert_V$. Since $(\phi^\star)^{-1}$ is a diffeomorphism from $W$ to $U_1^\star\t U_2^\star\subset X_Q\t(\R^k\t\C^{n-k})$, it follows that $\phi^{-1}$ is also a diffeomorphism near $V$ in $W$. Thus, making $V,W$ smaller if necessary, we can suppose that $U=\phi^{-1}(W)$ is open in $X_Q\t(\R^k\t\C^{n-k})=X_P$ and $\phi^{-1}:W\ra U$ is a diffeomorphism, with an inverse $\phi:U\ra W$ which is a diffeomorphism with an open set $W\subseteq X$. Hence $(P,U,\phi)$ is a g-chart on $X$, as in Theorem~\ref{bc3thm3}.

The construction of $\phi$ implies that if $p\in P$ corresponds to $(q,(a_{k+1},\ldots,a_n))$ under $P=Q\t\Z^{n-k}$ then $\phi_*(\la_pe^{i\th_p}):W\ra\C$ in Theorem \ref{bc3thm3} is
$\phi_*(\la_pe^{i\th_p})=\mu_qe^{i\th_q}\cdot e^{\sum_{j=k+1}^na_jz_j}$. Therefore equation \eq{bc3eq9} follows from \eq{bc3eq47}. This completes the proof of Theorem~\ref{bc3thm3}.

\addcontentsline{toc}{section}{References}
\renewcommand{\bibname}{References}
\printbibliography

\medskip

\noindent{\small\sc The Mathematical Institute, Radcliffe
Observatory Quarter, Woodstock Road, Oxford, OX2 6GG, U.K. 

\noindent E-mails:  {\tt hulya.arguz@maths.ox.ac.uk, dominic.joyce@maths.ox.ac.uk.}
}

\end{document}